\documentclass[11pt]{amsart}
\usepackage{amsmath, amssymb,amsthm}
\usepackage[hidelinks]{hyperref}
\usepackage[all]{xy}
\usepackage{enumerate,enumitem}
\usepackage{comment}
\usepackage{MnSymbol} 
\usepackage{mathrsfs}

\theoremstyle{plain}
\newtheorem{thm}{Theorem}[section]
\newtheorem{prop}[thm]{Proposition}
\newtheorem{lemma}[thm]{Lemma}
\newtheorem{cor}[thm]{Corollary}
\newtheorem{theoremletter}{Theorem}

\theoremstyle{definition}
\newtheorem{defn}[thm]{Definition}
\theoremstyle{remark}
\newtheorem{rem}[thm]{Remark}

\newcommand{\R}{\mathbb{R}}
\newcommand{\C}{\mathbb{C}}
\newcommand{\Q}{\mathbb{Q}}
\newcommand{\Z}{\mathbb{Z}}

\newcommand{\G}{\mathrm{G}}
\newcommand{\I}{\mathrm{I}}

\newcommand{\cD}{\mathcal{D}}
\newcommand{\cE}{\mathcal{E}}
\newcommand{\cF}{\mathcal{F}}
\newcommand{\cG}{\mathcal{G}}
\newcommand{\cO}{\mathcal{O}}
\newcommand{\cR}{\mathcal{R}}
\newcommand{\cT}{\mathcal{T}}
\newcommand{\cU}{\mathcal{U}}
\newcommand{\cV}{\mathcal{V}}
\newcommand{\cW}{\mathcal{W}}

\newcommand{\gm}{\mathfrak{m}}

\newcommand{\cC}{\mathscr{C}}
\newcommand{\cL}{\mathscr{L}}

\DeclareMathOperator{\ad}{ad} 

\DeclareMathOperator{\cusp}{cusp}
 
\DeclareMathOperator{\End}{End}
\DeclareMathOperator{\eis}{eis}
\DeclareMathOperator{\full}{full}
\DeclareMathOperator{\Frob}{Frob}

\DeclareMathOperator{\Gal}{Gal}
\DeclareMathOperator{\GL}{GL}
\DeclareMathOperator{\rH}{H}

\DeclareMathOperator{\Hom}{Hom}
\DeclareMathOperator{\Ext}{Ext}
\DeclareMathOperator{\nord}{n.ord}
\DeclareMathOperator{\ord}{ord}
\DeclareMathOperator{\red}{red}
\DeclareMathOperator{\rig}{rig}
\DeclareMathOperator{\res}{res}
\DeclareMathOperator{\ps}{ps}
\DeclareMathOperator{\tr}{tr}
\DeclareMathOperator{\univ}{univ}

\begin{document}

\title[weight $1$ Eisenstein points of the eigencurve]{On the failure of Gorensteinness at weight $1$ Eisenstein points of the eigencurve}
\author{Adel Betina, Mladen Dimitrov and Alice Pozzi}

\address{University of Vienna, Faculty of Mathematics,  Oskar-Morgenstern-Platz 1, 1090 Wien, Austria}
\email{adelbetina@gmail.com }

\address{University of Lille, CNRS, UMR 8524 -- Laboratoire Paul Painlev\'e, 59000 Lille, France}
\email{mladen.dimitrov@univ-lille.fr}

\address{Imperial College London, Department of Mathematics,  London SW7 2RH, United Kingdom}
\email{a.pozzi@imperial.ac.uk}

\thanks{The first author's acknowledges  support from the ERC  Horizon 2020 research and innovation programme (grant agreement n$^\circ$682152) and from the EPSRC (grant EP/R006563/1). The second author is partially supported by the ``Agence Nationale de la Recherche'' (grants ANR-11-LABX-0007-01, ANR-16-IDEX-0004 and ANR-18-CE40-0029). The third author has received partial support from 
ERC Consolidator Grant ``Euler systems and the Birch--Swinnerton--Dyer conjecture''. 
}

\begin{abstract}
We prove that the cuspidal  eigencurve $\cC_{\cusp}$ is \'etale over the weight space at any 
classical weight $1$ Eisenstein point $f$ and  meets two Eisenstein components of the eigencurve $\cC$  transversally  at $f$. Further, we prove that the local ring of $\cC$ at $f$ is
Cohen--Macaulay but not Gorenstein and compute the Fourier coefficients of a basis of overconvergent weight $1$  modular forms lying in the same   generalised eigenspace as $f$. In addition, we prove an $R=T$ theorem for the local ring at $f$ of the closed subspace of $\cC$ given by the union of $\cC_{\cusp}$ and one Eisenstein component and 
 prove unconditionally, via a geometric construction of the residue map, that  the corresponding congruence ideal is generated by the Kubota--Leopoldt $p$-adic $L$-function. Finally we obtain a new proof of the Ferrero--Greenberg Theorem and Gross' formula for the derivative 
of the  $p$-adic $L$-function  at the trivial zero.
\end{abstract}

\maketitle

\begin{small}
\tableofcontents
\end{small}

\section*{Introduction}

\addtocontents{toc}{\setcounter{tocdepth}{2}}

Let $p$ be any prime number and let  $N$ be a positive integer relatively prime to $p$. 
Let  $\cC$ be  the  reduced eigencurve of tame level $N$ endowed with a  flat and locally finite morphism 
$\kappa$ to the weight space $\cW$ (see \S\ref{basic-eigencurve}). 
The irreducible components of $\cC$ correspond either to Eisenstein or  to 
cuspidal  $p$-adic families of modular forms,  and the study of how these  meet plays a prominent 
role in the work  \cite{mazur-wiles} of Mazur and Wiles on the Iwasawa Main Conjecture for $\GL_1$ over $\Q$.
Any $p$-stabilisation of a newform of level $N$  has finite slope, thus defines a point in $\cC$ (see \cite[\S1.3]{bellaiche-Inv}).  The point  is  called {\it irregular} if the 
$p$-stabilisation is  unique,  and  {\it  regular} otherwise. The main theorem in \cite{bellaiche-dimitrov} describes the geometry of $\cC$  at all regular cuspidal weight $1$ points  ($\cC$ is smooth at such points and there is a precise condition for  $\kappa$ to be \'etale) and one expects the geometry at
 irregular cuspidal points to be more involved, because $\kappa$ is  never \'etale at such points and there are even examples  where $\cC$ is  not  smooth (see  \cite{dim-ghate}  and \cite{betina-dimitrov}).  
 
 In this paper  we study  the geometry of $\cC$ at classical weight $1$  Eisenstein points. 
Fix  a primitive  odd Dirichlet character $\phi$  of conductor $N$
 and consider the  Eisenstein series 
\begin{align}\label{eq:def-eisenstein}
 E_1(\mathbf{1}, \phi)(z)=  \frac{L(\phi,0)}{2} + \sum_{n \geqslant 1 } q^n\sum_{d \mid n } \phi(d)  \text{,  where }  q=e^{2i\pi z}.
 \end{align}
 It is a newform of level $N$ admitting the $p$-stabilisations 
 \begin{align*}
 E_1(\mathbf{1}, \phi)(z)-\phi(p)E_1(\mathbf{1}, \phi)(pz) \qquad  \text{and}\qquad  E_1(\mathbf{1}, \phi)(z)-E_1(\mathbf{1}, \phi)(pz), 
 \end{align*}
 of $U_p$-eigenvalues $1$ and $\phi(p)$, respectively, belonging to the Eisenstein components $\cE_{\mathbf{1}, \phi}$ and  $\cE_{\phi, \mathbf{1}}$. In particular, these two components
 intersect at weight $1$ if and only if $\phi(p)=1,$ 
 {\it i.e.} when  $ E_1(\mathbf{1}, \phi)$ is irregular. 
 If $\phi(p) \ne 1$, the constant term of  each one of these $p$-stabilisations is non-zero at some cusps in the multiplicative part of the ordinary locus of the modular curve  $X^{\rig}_{\mathrm{Iw}}$ of tame level $N$ and Iwahori level at $p$ (see \S\ref{basic-eigencurve}), hence such forms are not cuspidal-overconvergent and belong to a unique Eisenstein component.
 We thus restrict our attention to an Eisenstein series which is   irregular at $p$ and denote by $f$ its unique $p$-stabilisation.  
In addition to belonging to the  Eisenstein components $\cE_{\mathbf{1}, \phi}$ and  $\cE_{\phi, \mathbf{1}}$, the corresponding point   
also belongs to the cuspidal locus $\cC_{\cusp}$ of $\cC$, because  $f$  vanishes at all cusps of the multiplicative part of the ordinary locus of  $X^{\rig}_{\mathrm{Iw}}$ (see Proposition~\ref{q-expa family}).

Denote by  $\varLambda=\bar\Q_p \lsem X \rsem$ (resp. $\cT$) the completed strict local ring of $\cW$  at $\kappa(f)$ (resp. of $\cC$ at  $f$). The weight map $\kappa$  induces a finite, flat map $\varLambda \to \cT$ and a surjection 
$\varLambda[T_\ell,U_p]_{\ell\nmid Np}\twoheadrightarrow \cT$.

\begin{theoremletter} \label{main-thm} 
\begin{enumerate}
\item The  cuspidal eigencurve $\cC_{\cusp}$ is \'etale over $\cW$ at  $f$. In particular there exists a unique
cuspidal irreducible component $\cF$ of $\cC$ containing $f$. 

\item The  $\varLambda$-algebra  $\cT$ is  Cohen--Macaulay but not  Gorenstein, and is in fact isomorphic to 
\[\varLambda \times_{\bar\Q_p}\varLambda\times_{\bar\Q_p} \varLambda=\left\{(a,b,c)\in \varLambda^3 \mid a(0)=b(0)=c(0)\right\}\]
endowed with  $\varLambda$-algebra  structure via the  diagonal embedding. 

\item The image $\cT'$ of $\varLambda[T_\ell]_{\ell\nmid Np}$ in $\cT$  is a  complete intersection, and is  isomorphic to 
\[
\left\{(a,b,c)\in \varLambda \times_{\bar\Q_p} \varLambda\times_{\bar\Q_p} \varLambda \Big{|} 
\left(\cL(\phi^{-1})+\cL(\phi)\right)a'(0)=\cL(\phi^{-1})b'(0)+\cL(\phi)c'(0)  \right\},
\]
where  $\cL(\phi)$ is the $\cL$-invariant of  $\phi$ defined in \eqref{eq:l-inv}. 
\end{enumerate}
\end{theoremletter} 

The non-smoothness of  $\cC$ at $f$ is related to   the vanishing at $s=0$ of the Kubota--Leopoldt  $p$-adic $L$-function 
$L_p(\phi \omega_p,s)$,  $\omega_p$ being the Teichm\"uller character. 
 Using the relation between  the element $\zeta_\phi(X)\in \varLambda$ defined in \eqref{Kubota--leopoldt} and   the constant term of  $\cE_{\mathbf{1}, \phi}$,  Darmon, Dasgupta and Pollack constructed a
first order cuspidal deformation of $f$  which played a pivotal role in their work \cite{DDP}  on the Gross--Stark conjecture over  totally real number fields.  Their strategy consists in taking a suitable combination of Eisenstein series and, via the action of certain Hecke operators, producing  a cuspidal Hida family that, while not being an eigenform itself, yields a first order eigenform. 
Theorem~\ref{main-thm}, and its expected generalisation to totally real number fields where $p$ is inert (see \cite{betina-dimitrov-shih}),
gives a more precise result as it  implies the uniqueness of the cuspidal deformation at first and in fact at any order.

\begin{rem}
When $\phi$ is  quadratic then $\cF$  admits a familiar description as a  Hida family 
 interpolating  theta series for  the   imaginary quadratic field fixed by $\phi$.    Theorem~\ref{main-thm}(i) then implies that  
 the characteristic power series of the congruence ideal attached to $\cF$ (see \cite[(6.9)]{HT})  does not vanish at $X=0$, as conjectured by Hida and Tilouine \cite[p.~192]{HT}.  This is used in a forthcoming work of Burungale, Skinner and  Tian \cite{BST} on a conjecture of Perrin--Riou regarding the local non-triviality at $p$ of $p$-adic Beilinson-Kato elements attached to  elliptic curves over $\Q$. 
 Theorem~\ref{main-thm}(i) is also used in \cite{shih-wang} to show that 
  the source of the map $\varpi$ in Sharifi's Conjecture \cite[Proposition~5.7]{sharifi}  (see also \cite[(2.3)]{shih-wang})
   is one-dimensional after localising at a prime ideal of the Iwasawa algebra corresponding to a trivial zero of the Kubota--Leopoldt $p$-adic $L$-function.  
 \end{rem}

 Our approach to the geometry of the eigencurve is Galois theoretic, and consists in
proving a modularity theorem for the  ordinary deformations of the Galois representation attached  to $f$. However, the reducibility of the latter causes several issues. First, since the Artin representation $\mathbf{1} \oplus \phi$ attached  to $f$ is  decomposable,   its deformation functor is not representable  in the sense of Mazur \cite{mazur}. Moreover, 
the irregularity assumption  implies that $\mathbf{1} \oplus \phi$ is  trivial on the decomposition group $\G_{\Q_p}$, creating an obstacle to imposing an ordinary deformation condition at $p$. 
In order to circumvent these  difficulties we introduce two  reducible indecomposable representations 
\[ \rho= \begin{pmatrix} \phi& \eta\\ 0 & \mathbf{1} \end{pmatrix}:\G_{\Q}\to \GL_2(\bar\Q_p)\text{  and  } 
\rho'= \begin{pmatrix}  \mathbf{1}& \phi\eta'\\ 0 & \phi \end{pmatrix}:\G_{\Q}\to \GL_2(\bar\Q_p), \]
where $[\eta]$ and $[\eta']$ are bases  of the lines $\rH^1(\Q, \phi)\simeq \Ext^1_{\Q}(\mathbf{1}, \phi)$ and $\rH^1(\Q, \phi^{-1})\simeq \Ext^1_{\Q}( \phi, \mathbf{1})$. The representations $\rho$ and  $\rho'$ are, up to isomorphism, the unique non-decomposable extensions  and  admit  universal ordinary deformation rings  
$\cR^{\ord}_{\rho}$ and  $\cR^{\ord}_{\rho'}$, respectively.
Finally we introduce a universal  ring $\cR_{\cusp}$ classifying pairs of ordinary deformations of $(\rho,\rho')$ sharing the  same traces and  Frobenius action on the unramified  $\G_{\Q_p}$-quotient. 
A substantial  part of this paper is devoted to the computation of their   tangent spaces relying on  
Baker--Brumer's Theorem~\cite{brumer} in Transcendence Theory. The vanishing of the relative tangent space of $\cR_{\cusp}$ allows us to deduce  as in \cite{bellaiche-dimitrov},  that $\cR_{\cusp}$  is unramified over $\varLambda$ and isomorphic to  the completed strict local ring $\cT_{\cusp}$  of $\cC_{\cusp}$ at $f$. 
Whereas the $\cT$-valued two dimensional pseudo-character of $\G_{\Q}$ does not arise from 
an ordinary  representation, because of the non-Gorensteinness of $\cT$, we manage to prove a
modularity result for $\cR^{\ord}_{\rho}$, and  for $\cR^{\ord}_{\rho'}$ (see \S\ref{s:ord-modularity}).  Denote by  $\cT^{\ord}_\rho$ the completed strict local ring at $f$ of the reduced equidimensional closed analytic subspace of $\cC$ given by the union of $\cC_{\cusp}$ and the component of $\cC$ corresponding to $\cE_{\mathbf{1},\phi}$ (see \S\ref{basic-eigencurve}).

\begin{theoremletter}
The $\varLambda$-algebra $\cR^{\ord}_{\rho}$ is a complete  intersection isomorphic to   $\cT^{\ord}_\rho$.
\end{theoremletter}

The Gorenstein property  plays a prominent role in the theory of Hecke algebras, as it guarantees the freeness 
of the module of modular forms over those algebras. Theorem~\ref{main-thm} thus provides a 
testing ground for challenging questions in Iwasawa Theory, such as the construction of $p$-adic $L$-functions in a neighbourhood  of a non-Gorenstein point of $\cC$, and a formulation of a Main Conjecture at such a point.  It further  suggests that this remarkable phenomenon is related to the action of the $U_p$ 
operator, as the $p$-deprived Hecke algebra $\cT'$ is Gorenstein and even a  complete intersection (see Corollary~\ref{generalisedeigenspace}).

Let $J_{\eis} \subset \cT_{\cusp}$ be the Eisenstein ideal associated to $\cE_{\mathbf{1}, \phi}$ (see \S.\ref{subsection: The Eisenstein ideal}). 

\begin{theoremletter}\label{congridealintro} There exists an isomorphism of $\varLambda$-algebras $\cT_{\cusp}/J_{\eis} \xrightarrow{\sim} \varLambda/(\zeta_\phi(X))$. 
\end{theoremletter}

In the absence of a trivial zero ({\it i.e.}  if $\phi(p) \ne 1$), this is a well known result of Mazur-Wiles and Wiles \cite[Theorem~4.1]{wiles-imc} (see also Ohta \cite{ohta-eis} and Emerton \cite{emerton-eis}  when $p\geqslant 5$). 
Our proof uses a geometric residue map from  the space of Hida families onto the ordinary cuspidal group whose  kernel consists of  cuspidal Hida families. 
Its definition taps into  Pilloni's  geometric constructions \cite{pilloni} of $p$-adic families of modular forms, hence differs from Ohta's residue map. The surjectivity of the  residue map is deduced from the fact that the ordinary locus of the modular curve is an affinoid (see Proposition~\ref{residuemapfamily}), while the Kubota--Leopoldt $p$-adic $L$-function appears in 
the constant terms of the Eisenstein family $\cE_{\mathbf{1}, \phi}$ (see Proposition~\ref{q-expa family}). 

Combining Theorems~\ref{main-thm} and \ref{congridealintro} yields a new proof  of the famous result of Ferrero--Greenberg  \cite{ferrero-greenberg} and Gross--Koblitz on  the non-vanishing of  $L'_p(\phi \omega_p,0)$ (see Proposition~\ref{ordL_p}).

The failure of \'etaleness for the eigencurve at $f$, combined with  the perfectness of the duality between $\cT$ and 
the $\varLambda$-module of ordinary families specialising to  $f$ established in \S\ref{sec:duality},  implies
the existence of non-classical forms in the generalised eigenspace associated to $f$. 
Their $q$-expansions  admit an explicit description in terms of $p$-adic logarithms of rational numbers and of $p$-units of the splitting field of $\phi$; this  contrasts with the  cuspidal regular setting  in the work of Darmon, Lauder and Rotger {\cite{DLR3}} where only $\ell$-units for $\ell\ne p$ are involved.

\begin{theoremletter} \label{qexpansionoverc}
Let $M^{\dag}\lsem f \rsem $ (resp. $S^{\dag}\lsem f \rsem$)  be the generalised eigenspace attached to $f$ inside the space of weight $1$ overconvergent  modular forms (resp. cuspforms) of tame level $N$ and central character $\phi$. 
Then  $S^{\dag}\lsem f \rsem=\bar\Q_p f$, while a complement of $S^{\dag}\lsem f \rsem$ in 
$M^{\dag}\lsem f \rsem$ is spanned by 
\begin{align*}
f^\dag_{\phi, \mathbf{1}}& =\sum_{n \geqslant 1 } q^n \sum_{d|n,  \, p\nmid d} \phi(d) 
\left(\ord_p(n) \cL(\phi)+\log_p\left(\tfrac{d^2}{n}\right)\right) \text{ and } \\ 
f^\dag_{\mathbf{1}, \phi}&=(\cL(\phi)+\cL(\phi^{-1}))\frac{L(\phi, 0)}{2}+
\sum_{n \geqslant 1 } q^n \sum_{d|n, \, p\nmid d} \phi(d) 
\left(\ord_p(n) \cL(\phi^{-1})-\log_p\left(\tfrac{d^2}{n}\right)\right),\end{align*}
where $\ord_p$ is the $p$-adic valuation and $\log_p$ is the $p$-adic logarithm normalised by $\log_p(p)=0$. 
\end{theoremletter}

Determining the coefficients of the classical (non $p$-stabilised) form  $E_1(\mathbf{1}, \phi) \in M^{\dag}\lsem f \rsem $ in the basis $\{f,f^\dag_{\phi, \mathbf{1}}, f^\dag_{\mathbf{1}, \phi} \}$ yields a new proof (see Corollary~\ref{Grossstarkcor}) of  Gross' formula  \cite{gross} for the derivative at a trivial zero: 
\[L'_p(\phi \omega_p, 0)= \cL(\phi)L(\phi, 0).\]  

For ``tame''  analogues of these results the  interested reader is referred to Remark~\ref{tame} where  we illustrate a 
rather striking analogy with a phenomenon arising in Mazur's Eisenstein ideal setting  \cite{mazur-eis}, as well as to 
the recent work of P.~Wake \cite{wake} where such phenomena  are related to the notion of ``extra reducibility".

{\it Acknolwedgements.} 
    We are mostly indebted to   D.~Benois and H.~Darmon for numerous stimulating discussions that helped this project  emerge.  We would also like to thank J.~Bella{\"{\i}}che, T.~Berger, A.~Burungale, A.~Lauder,   E.~Lecouturier,  P. ~Kassaei, V.~Rotger, S.-C.~Shih, P.~Wake and C.~Wang-Erickson for their interest and  helpful comments.  Finally, we would like to thank the anonymous referee  for his or her careful review of the manuscript and for all the remarks and suggestions which helped us improve the quality of the exposition.

\section{Ordinary Galois deformations}\label{deformation}\
For a  perfect field $L$ we denote  $\G_L=\Gal(\bar{L}/L)$ its absolute Galois group.  Given a prime number $\ell$, 
we  fix  an embedding $\iota_\ell: \bar{\Q} \hookrightarrow \bar\Q_\ell$ which determines an embedding $\G_{\Q_\ell}\hookrightarrow \G_{\Q}$, and denote by $\I_{\Q_\ell}$ the inertia subgroup at $\ell$ and by $\Frob_\ell\in\G_{\Q_\ell} $ an arithmetic Frobenius. We also fix  an embedding $\iota_\infty: \bar{\Q} \hookrightarrow \C$
which  determines a complex conjugation $\tau\in \G_{\Q}$. 
 All Galois representations are with coefficients in $\bar\Q_p$ unless stated otherwise.

\subsection{A canonical reducible non-split  Galois representation attached to \texorpdfstring{$\phi$}{}}\label{def-eta}
Consider the  unique element $\eta_{\mathbf{1}}\in \rH^1(\Q, \bar\Q_p) $ whose restriction  to 
the image of  $\rH^1(\Q_p,\bar\Q_p)$ in $\rH^1(\I_{\Q_p}, \bar\Q_p)$
corresponds via  Local Class Field Theory to the $p$-adic logarithm $\log_p$. By Global Class Field Theory one has 
 $\eta_{\mathbf{1}}(\Frob_\ell) =-\log_p(\ell)$ for all primes $\ell\neq p$.
As $\phi$ is odd, one has $\dim_{\bar\Q_p} \rH^1(\Q, \phi)=1$ (see for example \cite[(8)]{bellaiche-dimitrov}). 
Noting that the restriction map $\rH^1(\Q, \phi)\to  \rH^1(\I_{\Q_p}, \bar\Q_p )$ is injective, we 
let $[\eta]\in \rH^1(\Q, \phi) $ be the unique element mapping to  $\log_p$.

Given any  cocycle $\eta'': \G_{\Q}\to \phi$ representing  a non-trivial element  $[\eta'']\in \rH^1(\Q, \phi)$,  the $\G_{\Q}$-representations  $\left( \begin{smallmatrix} \phi& \eta\\ 0 & \mathbf{1} \end{smallmatrix}\right)$ and
$\left( \begin{smallmatrix} \phi& \eta''\\ 0 & \mathbf{1} \end{smallmatrix}\right)$ are always conjugated by an upper-triangular matrix which is unipotent if and only if $[\eta]=[\eta'']$. 
Since $\phi-\mathbf{1}$ is a basis of the coboundaries and $\phi_{|\G_{\Q_p}}=\mathbf{1}$, the restriction  $\eta_{|\G_{\Q_p}}$
only depends on $[\eta]$.  Moreover, as $\phi(\tau)=-1$, there exists a unique 
cocycle $\eta$  representing the cohomology class $[\eta]$  such that $\eta(\tau)=0$. 
We let $\rho=\left( \begin{smallmatrix} \phi& \eta\\ 0 & \mathbf{1} \end{smallmatrix}\right):\G_\Q\to \GL_2(\bar\Q_p)$.

\subsection{Ordinary deformations of \texorpdfstring{$\rho$}{}}\label{ord-def}

Let $\mathscr{A}$ be the category of  Artinian local $\bar\Q_p$-algebras $A$ with 
maximal ideal $\gm_A$ and residue field $\bar{\Q}_ p$, where the morphisms are local homomorphisms of 
$\bar\Q_p$-algebras inducing identity on the residue field (note that 
$A/\gm_A=\bar{\Q}_ p$ canonically as  $\bar\Q_p$-algebras).

Consider the functor $\cD^{\univ}_{\rho}$ assigning to $A\in  \mathscr{A}$ the set of lifts  $\rho_A:\G_{\Q}\to \GL_2(A)$ of $\rho$  
modulo strict equivalence ({\it i.e.} conjugation by an element of $1+\mathrm{M}_2(\gm_A)$). 
Since $\phi$ and $[\eta]$ are both non-trivial, the centraliser of the image of $\rho$ consists only of scalar matrices, 
hence $\cD^{\univ}_{\rho}$  is pro-representable by a complete Noetherian local $\bar\Q_p$-algebra
$\cR^{\univ}_{\rho}$, called the universal deformation ring (see  \cite{mazur}).

Denote by $V_{\rho}=\bar\Q_p^2$ the representation space of $\rho$ endowed with its canonical basis $(e_1,e_2)$. There exists a unique $\G_{\Q_p}$-stable (in fact, it is $\G_{\Q}$-stable) filtration with unramified quotient
 \begin{align}\label{filtration-rho}
0 \to F_{\rho}=\bar\Q_p e_1\to V_{\rho}\to V_{\rho}/F_{\rho} \to 0.
\end{align}

\begin{defn}\label{def-ord-functor}
The functor $\cD^{\ord}_{\rho}$   assigns to $A\in  \mathscr{A}$ the set of tuples $(\rho_A,F_A)$, where
\begin{enumerate}
\item $\rho_A:\G_{\Q}\to \GL_2(A)$ is a continuous representation, such that $\rho_A\mod \gm_A = \rho$, and 
\item $F_A \subset A^2$ is a free direct factor  over $A$ of rank $1$  which is $\G_{\Q_p}$-stable and 
such that $\G_{\Q_p}$ acts  on $A^2/ F_A$ by an unramified character, denoted $\chi_A$, 
 \end{enumerate}
modulo strict equivalence relation  $[(\rho_A,F_A)]=[(P\rho_AP^{-1},P\cdot F_A)]$ for $P\in 1+\mathrm{M}_2(\gm_A)$. 
\end{defn}

As the restriction of $[\eta]$ to  $\I_{\Q_p}$ is non-trivial, 
it follows from  Schlessinger's criterion (see \cite[Corollary 6.6]{gouvea})  that $\cD^{\ord}_{\rho}$ is pro-representable  
by a quotient  of $\cR^{\univ}_{\rho}$. We will now provide an explicit description of the ideal  defining that quotient. 

 Choose a lift  $\rho_{\univ}=\left(\begin{smallmatrix} a&b\\ c& d\end{smallmatrix}\right): \G_\Q\to \GL_2(\cR^{\univ}_{\rho})$ representing the universal deformation of $\rho$ and define 
 $\cR^{\ord}_{\rho}$ as the quotient of  $\cR^{\univ}_{\rho} \lsem Y\rsem$ by the ideal
\[\left (d(h)-1-b(h)Y, c(g)+(d(g)-a(g))Y-b(g)Y^2; h\in \I_{\Q_p}, g\in\G_{\Q_p}\right).\]
Choosing any $h_0\in  \I_{\Q_p}$ such that $\eta(h_0)\neq 0$ 
(so that  $b(h_0)\in (\cR^{\univ}_{\rho})^{\times}$), the linear relation  $d(h_0)-1-b(h_0)Y=0$ shows that the natural composed map 
$\cR^{\univ}_{\rho}\to \cR^{\ord}_{\rho}$ is surjective.  By construction, the push-forward of $\rho_{\univ}$ along that surjection together with the line
having basis  $e_1+Y e_2$,   yield a point of 
$\cD^{\ord}_{\rho}(\cR^{\ord}_{\rho})$. Conversely, any 
point of $\cD^{\ord}_{\rho}(A)$ is represented by a push-forward  $\rho_A$ of $\rho_{\univ}$ along a (unique) homomorphism 
$\varphi_A: \cR^{\univ}_{\rho}\to A$,  and an ordinary line $F_A\subset A^2$. 
The latter has  a basis  $e_1+y e_2$ with $y\in \gm_A$, because $F_A \otimes_{A} \bar\Q_p = F_{\rho}$. 
Let $\widetilde\varphi_A: \cR^{\univ}_{\rho}\lsem Y\rsem \to A$ be the homomorphism extending $\varphi_A$ and sending $Y$ to $y$.  As
 \begin{align}\label{Y-line}
\begin{pmatrix}
1&0\\ -y& 1\end{pmatrix} \begin{pmatrix} a&b\\ c& d\end{pmatrix}\begin{pmatrix} 1&0\\ y& 1\end{pmatrix}=
\begin{pmatrix} a+by & b \\ c+(d-a)y-by^2 & d-by\end{pmatrix}
\end{align}
is the conjugation of $\rho_A$ in an ordinary basis,  it follows that $\widetilde\varphi_A$ factors through $\cR^{\ord}_{\rho}$, hence 
 $\varphi_A$ factors  through $\cR^{\ord}_{\rho}$ as well. Therefore $\cR^{\ord}_{\rho}$  represents $\cD^{\ord}_{\rho}$,  in particular the kernel of the natural surjection $\cR^{\univ}_{\rho}\to \cR^{\ord}_{\rho}$
 is independent of the particular choice of $\rho_{\univ}$.   
It follows  that any  tuple $(\rho_A,F_A)$ in $\cD^{\ord}_{\rho}(A)$ is characterised  by $\rho_A$ alone, {\it i.e.}  when  the ordinary filtration of $\rho_A$ exists, then it is unique. For this reason, and as the unramified character $\chi_A$  plays an important role, we  will sometimes denote  a point in $\cD^{\ord}_{\rho}(A)$ by  $(\rho_A,\chi_A)$.

One can define the nearly-ordinary deformation  functor $\cD^{\nord}_{\rho}$  by using the same definition 
as for $\cD^{\ord}_{\rho}$, but without imposing  the $\G_{\Q_p}$-quotient  to be unramified. An argument similar to the one presented above 
shows that   $\cD^{\nord}_{\rho}$ is  pro-representable by a quotient $\cR^{\nord}_{\rho}$ of $\cR^{\univ}_{\rho}\lsem Y \rsem$, and  $\cR^{\nord}_{\rho}$ is generated over $\mathrm{Im}(\cR^{\univ}_{\rho} \to \cR^{\nord}_\rho)$ by a root of 
the polynomial $b(h_0)Y^2+(a(h_0)-d(h_0))Y-c(h_0)$, or equivalently a root of 
$U^2-\tr(\rho_{\univ})(h_0)U+\det(\rho_{\univ})(h_0)$. 
Using Theorem~\ref{main-thm}, one can see that $\cR^{\nord}_\rho$ is indeed quadratic over 
$\mathrm{Im}(\cR^{\univ}_{\rho} \to \cR^{\nord}_\rho)$.

Finally let  $\cD^{\ord}_{\rho,0}$ be the sub-functor of $\cD^{\ord}_{\rho}$  given by the deformations with fixed determinant equal to $\phi$. Since $\varLambda$ is the universal deformation ring of $\phi$ (see \cite[\S6]{bellaiche-dimitrov}) the natural transformation 
$\rho_A\mapsto \det(\rho_A)$ endows $\cR^{\ord}_{\rho}$ with a natural structure of a $\varLambda$-algebra and
$\cD^{\ord}_{\rho,0}$ is pro-representable by $\cR^{\ord}_{\rho,0}=\cR^{\ord}_{\rho}/\gm_{\varLambda}\cR^{\ord}_{\rho}$.

\subsection{Reducible deformations of \texorpdfstring{$\rho$}{}}\label{red-def}

\begin{defn} 
Let  $\cD^{\red}_{\rho}$ be the subfunctor of $\cD^{\ord}_{\rho}$ consisting of $\G_{\Q}$-reducible deformations. 
\end{defn}

\begin{lemma}
The functor $\cD^{\red}_{\rho}$ is pro-representable by a quotient $\cR^{\red}_{\rho}$ of $\cR^{\ord}_{\rho}$. 
\end{lemma}
\begin{proof}  Choose a lift $\left(\begin{smallmatrix} a&b\\ c& d\end{smallmatrix}\right): \G_\Q\to \GL_2(\cR^{\ord}_{\rho})$ representing  the universal ordinary deformation   sending  the complex conjugation $\tau$  to $\left( \begin{smallmatrix}  -1& 0\\  0 & 1 \end{smallmatrix} \right)$. Applying \eqref{Y-line}  to a $\G_{\Q}$-stable line with  basis $e_1+ y e_2$, one finds that
 $c+(d-a)y-by^2=0$. Evaluating at  $\tau\in \G_{\Q}$ yields $y=0$ and  shows that
 \[\cR^{\red}_{\rho}=\cR^{\ord}_{\rho}/(c(g); g\in \G_{\Q}).\qedhere\]
 \end{proof}
 
While  $\rho$ admits a $\G_\Q$-quotient which is unramified at $p$ (see \eqref{filtration-rho}),  
this is not necessarily  true for all its ordinary, reducible lifts. To account for this discrepancy, we introduce the following functor. 

\begin{defn} \label{defn:eis}
Let $\cD^{\eis}_{\rho}$ be the subfunctor of $\cD^{\red}_{\rho}$ consisting of deformations which are 
reducible and ordinary for the same filtration, {\it i.e.} admit a rank $1$  $\G_{\Q}$-quotient which is unramified at $p$. 
\end{defn}

\begin{lemma}\label{idealJ=C}
The functor $\cD^{\eis}_{\rho}$ is pro-representable by  $\cR^{\eis}_{\rho}=\cR^{\ord}_{\rho}/(C(g); g\in \G_{\Q})$, 
where $\left(\begin{smallmatrix} A& B\\ C&D\end{smallmatrix}\right): \G_\Q\to \GL_2(\cR^{\ord}_{\rho})$  is a lift representing the universal ordinary deformation  in an ordinary basis. 
 
\end{lemma}

Finally,  let  $\cD^{\red}_{\rho,0}$ and $\cD^{\eis}_{\rho,0}$ be the sub-functors of  $\cD^{\red}_{\rho}$ and $\cD^{\eis}_{\rho}$, respectively,  classifying  the deformations having fixed determinant equal to $\phi$. By the discussion in \S\ref{ord-def} they are
pro-representable by  $\cR^{\red}_{\rho,0}=\cR^{\red}_{\rho}/\gm_{\varLambda}\cR^{\red}_{\rho}$ and 
$\cR^{\eis}_{\rho,0}=\cR^{\eis}_{\rho}/\gm_{\varLambda}\cR^{\eis}_{\rho}$,  respectively. 
We will  see in \S\ref{tg dimension all}  that $\cD^{\eis}_{\rho}$ and  $\cD^{\red}_{\rho}$ differ.

\subsection{Cuspidal deformations of \texorpdfstring{$\rho$}{}}
We will exploit the interchangeability of the characters $\mathbf{1}$ and $\phi$ used in the definition of 
$\cD^{\ord}_\rho$  to define another deformation functor denoted  $\cD_{\cusp}$ (together with its relative version $\cD_{\cusp}^0$),   and show that it is  pro-representable by a universal deformation ring  $\cR_{\cusp}$. We will later show that $\cR_{\cusp}$ is isomorphic to the local ring of the cuspidal eigencurve at $f$, thus justifying the notation.

Recall that in \S\ref{def-eta} we fixed a basis  $[\eta] \in \rH^1(\Q, \phi)$ and constructed  a  
 representation $\rho= \left(\begin{smallmatrix} \phi & \eta \\ 0 &\mathbf{1} \end{smallmatrix}\right)$. 
Since  $\dim_{\bar\Q_p} \rH^1(\Q, \phi^{-1})=1$ we can perform the following  analogous construction.
Fix a basis  $[\eta'] \in \rH^1(\Q, \phi^{-1})$ such that $[\eta']_{\mid \I_{\Q_p}} $ corresponds via Local Class Field Theory to the $p$-adic logarithm and let $\eta'$ be the unique  representative  such that $\eta'(\tau)=0$. 
Let   $\rho'=\left(\begin{smallmatrix}\mathbf{1} & \phi \eta' \\0 &\phi\end{smallmatrix}\right):\G_\Q\to \GL_2(\bar\Q_p)$ and 
consider, as in Definition~\ref{def-ord-functor}, the  functor   $\cD^{\ord}_{\rho'}$ which is analogously pro-representable  by a universal deformation algebra  $\cR^{\ord}_{\rho'}$.

\begin{defn}
Let $\cD_{\cusp}$ be the functor assigning to  $A \in \mathscr{A}$ the set of equivalence classes of pairs $((\rho_A, \chi_A),(\rho'_A, \chi'_A))$ in $\cD^{\ord}_{\rho}(A)\times \cD^{\ord}_{\rho'} (A)$ such that 
\begin{enumerate}
\item  $\tr(\rho_A)=\tr(\rho'_A )$,  $\det(\rho_A)=\det(\rho'_A )$, and 
\item  $\chi_A(\Frob_p)= \chi'_A(\Frob_p)$. 
\end{enumerate}
\end{defn}

Put  $\rho_{\cR}= \rho_{\cR^{\ord}_{\rho}}$,   $\chi_{\cR}= \chi_{\cR^{\ord}_{\rho}}$, and 
analogously $\rho'_{\cR}= \rho_{\cR^{\ord}_{\rho'}}$,   $\chi'_{\cR}= \chi_{\cR^{\ord}_{\rho'}}$. 
By definition the functor $\cD_{\cusp}$ is pro-representable by the quotient $\cR_{\cusp}$
 of $ \cR^{\ord}_{\rho} \widehat{\otimes}_{\varLambda} \cR^{\ord}_{\rho'}$ by the  ideal 
\[ \left(\tr (\rho_{\cR})(g) \otimes 1 - 1\otimes \tr(\rho'_{\cR})(g),
\chi_{\cR}(\Frob_p)\otimes 1 - 1\otimes \chi'_{\cR}(\Frob_p); g\in \G_{\Q}
\right). \]

\begin{lemma}\label{lem-surj}
The natural homomorphisms $\cR^{\ord}_{\rho} \to \cR_{\cusp}$ and $\cR^{\ord}_{\rho'} \to \cR_{\cusp}$ are surjective. 
\end{lemma}

\begin{proof}
Let $\cR^{\ps}$ be the universal ring pro-representing deformations of the pseudo-character $\phi +\mathbf{1}$ (see \cite[Lemma 1.4.2]{kisin}). Since $\tr(\rho_{\univ})$ is a pseudo-character lifting $\phi +\mathbf{1}$, the universal property gives a 
homomorphism  $\cR^{\ps} \to \cR^{\univ}_\rho$  which is surjective by \cite[Corollary~1.4.4(ii)]{kisin}. 
Composing with the surjection  $\cR^{\univ}_\rho\twoheadrightarrow \cR^{\ord}_\rho$ 
yields a natural surjection $\cR^{\ps} \twoheadrightarrow \cR^{\ord}_\rho$. It follows that  $\gm_{\cR^{\ord}_{\rho}}$  is generated by the values of $\tr (\rho_{\cR})-\phi-\mathbf{1}$ and similarly for $\gm_{\cR^{\ord}_{\rho'}}$. Using this and the fact that 
$(\tr(\rho_{\cR}) -\phi-\mathbf{1})\otimes 1$ and  $1\otimes (\tr(\rho'_{\cR})-\phi-\mathbf{1})$ have the same image under   $ \cR^{\ord}_{\rho} \widehat{\otimes}_{\varLambda} \cR^{\ord}_{\rho'}  \twoheadrightarrow  \cR_{\cusp}$, one can show that 
$\gm_{\cR^{\ord}_{\rho}} \cR_{\cusp} = \gm_{ \cR_{\cusp}}$. In other terms,  the natural homomorphisms  $\cR^{\ord}_{\rho} \to \cR_{\cusp}$ and $\cR^{\ord}_{\rho'} \to \cR_{\cusp}$ are unramified morphisms of complete local Noetherian rings having the same residue field,  hence they are surjective.
\end{proof}

Let $\cD_{\cusp}^0$  be the subfunctor of $\cD_{\cusp}$ consisting  of deformations with fixed determinant equal to $\phi$. 
It is pro-representable by $\cR_{\cusp}^0=\cR_{\cusp}/\gm_{\varLambda}\cR_{\cusp}$.

\section{Tangent spaces}\label{sec:tangent}

 In this section we interpret the tangent spaces of the functors introduced in \S\ref{deformation}
  using  Galois cohomology and   compute their dimensions. We will discover that all infinitesimal reducible deformations  of $\rho$ are necessarily ordinary.

\subsection{Tangent spaces for nearly ordinary deformations}\label{tg dimension all}
 
 Let $\bar\Q_p[\epsilon]$ denote the $\bar\Q_p$-algebra of dual numbers.
Recall that there is a  natural isomorphism: 
\begin{align}
 \rH^1(\Q, \ad(\rho)) \xrightarrow{\sim} \cD^{\univ}_{\rho}(\bar\Q_p[\epsilon])=t^{\univ}_{\rho}, \,\,[\left( \begin{smallmatrix} a& b\\ c & d \end{smallmatrix}\right)] \mapsto [\rho_\epsilon], \text{ where } \rho_\epsilon=(1+\epsilon \left( \begin{smallmatrix} a& b\\ c & d \end{smallmatrix}\right)) \rho,
\end{align}
identifying   $\rH^1(\Q, \ad^0(\rho))$  with $t^{\univ}_{\rho,0}$, where $\ad(\rho)$ (resp. $\ad^0(\rho)$) is the adjoint representation 
of $\rho$ (resp. the sub-representation on trace $0$ elements $\End^0_{\bar\Q_p}(V_{\rho})$). Hence the tangent spaces 
\begin{align}
t^{\ord}_{\rho}=\cD^{\ord}_{\rho}(\bar\Q_p[\epsilon]), \
t^{\red}_{\rho}=\cD^{\red}_{\rho}(\bar\Q_p[\epsilon]),  \text{ and } 
t^{\eis}_{\rho}=\cD^{\eis}_{\rho}(\bar\Q_p[\epsilon]).
\end{align}
of the functors defined in \S\ref{ord-def}  and \S\ref{red-def} are  naturally isomorphic to subspaces of $\rH^1(\Q, \ad(\rho))$ that we will now determine precisely. 

Let  $W_{\rho}$ be the kernel of the natural homomorphism $\End_{\bar\Q_p}(V_{\rho}) \to \Hom_{\bar\Q_p}(F_{\rho}, V_{\rho}/F_{\rho})$ of $\G_\Q$-representations arising from  \eqref{filtration-rho} and   let  $W_\rho^0=W_\rho\cap \End^0_{\bar\Q_p}(V_{\rho})$.
Let $W'_{\rho}$ be the kernel of the natural homomorphism $W_{\rho}\to \Hom_{\bar\Q_p}(V_{\rho}/F_{\rho}, V_{\rho}/F_{\rho})$
and   let  $W'_\rho{}^0=W'_\rho\cap \End^0_{\bar\Q_p}(V_{\rho})$.

The basis $(e_1,e_2)$ of $V_{\rho}$ in which $\rho=\left( \begin{smallmatrix} \phi& \eta\\ 0 & \mathbf{1} \end{smallmatrix}\right)$
yields an identification   $\End_{\bar\Q_p}(V_{\rho})=\mathrm{M}_2(\bar\Q_p)$ under which $W_\rho$ (resp. $W'_\rho$) corresponds to  the subspace of the upper triangular matrices (resp.  matrices of the form $\left( \begin{smallmatrix} a& b\\ 0 & 0 \end{smallmatrix}\right)$).

\begin{prop} \label{lemmatd}
One has  $t^{\nord}_{\rho} \simeq \rH^1(\Q,{W_\rho})\oplus \bar\Q_p$,  $t^{\eis}_{\rho}\simeq \rH^1(\Q,{W'_\rho})$,  
$t^{\red}_{\rho}=t^{\ord}_{\rho}\simeq \rH^1(\Q,{W_\rho})$. Moreover 
$t^{\nord}_{\rho,0} \simeq \rH^1(\Q,{W^0_\rho})\oplus \bar\Q_p$ and 
$t^{\red}_{\rho,0}=t^{\ord}_{\rho,0}\simeq \rH^1(\Q,{W^0_\rho})$. Finally $t^{\eis}_{\rho,0}=t^{\eis}_{\rho}\cap t^{\ord}_{\rho,0}=\{0\}$.  
\end{prop}

 \begin{proof}  
  An infinitesimal deformation $[\rho_\epsilon]$  of $\rho$ can always be   represented by a lift 
 \begin{align}\label{rho-epsilon}
 \rho_\epsilon=\left(1+\epsilon \left(\begin{matrix}a & b \\ c & d \end{matrix}\right) \right)\rho=
\left( \begin{matrix} \phi(1+\epsilon a) & \epsilon b +\eta(1+\epsilon a)\\ \epsilon\phi c & 1+\epsilon d+\epsilon\eta c \end{matrix}\right), \end{align}
such that  $\rho_\epsilon(\tau)=\left(\begin{smallmatrix}-1 & 0 \\ 0 & 1 \end{smallmatrix}\right)$,   where $ \left(\begin{smallmatrix}a & b \\ c & d \end{smallmatrix}\right):\G_{\Q} \to  \mathrm{M}_2(\bar\Q_p)$ is a cocycle (we recall that $\mathrm{M}_2(\bar\Q_p)$ is endowed with the adjoint action of $\rho$). 
 As changing the lift  amounts to changing the cocycle 
$ \left(\begin{smallmatrix}a & b \\ c & d \end{smallmatrix}\right) $ by a coboundary, one sees from 
\begin{align}\label{rho-adjoint}
 \rho \left(\begin{matrix}a & b \\ c & d \end{matrix}\right)\rho^{-1}=
\left( \begin{matrix} a+ c\eta\phi^{-1}&  b\phi +(d-a)\eta- c\eta^2\phi^{-1}\\ c\phi^{-1} &  d - c\eta\phi^{-1}  \end{matrix}\right).
 \end{align}
that the cocycle $c:\G_\Q\to \bar\Q_p(\phi^{-1}) $ is changed by a coboundary as well, hence $[c]\in \rH^1(\Q, \phi^{-1})$ is uniquely determined 
by $[\rho_\epsilon]\in t^{\univ}_{\rho}$. As $\rH^0(\Q,\phi^{-1})=\{0\}$, the exact sequence of $\G_{\Q}$-modules
\[ 0 \to W_\rho \to \ad \rho \to \phi^{-1} \to 0,\]
where the  map $\ad \rho\to\phi^{-1}$   is given by  $\left[\left(\begin{smallmatrix}a & b \\ c & d \end{smallmatrix}\right) \right]\mapsto[c]$,  yields an exact sequence in cohomology
\begin{align}\label{exact-adjoint}
0 \to \rH^1(\Q,W_\rho) \to \rH^1(\Q,\ad \rho) \to \rH^1(\Q,\phi^{-1})\to \rH^2(\Q,W_\rho).
\end{align}
 
 As $\rho_{|\G_{\Q_p}}$ is indecomposable, any $\rho_\epsilon(\G_{\Q_p})$-stable   $\bar\Q_p[\epsilon]$-line 
 in $\bar\Q_p[\epsilon]^2$ has basis $e_1+ \epsilon \mu \cdot e_2$  for some $\mu\in \bar\Q_p$. By \eqref{rho-epsilon}, for $L=\Q$ or $\Q_p$ one has 
\[\bar\Q_p[\epsilon](e_1+ \epsilon \mu \cdot e_2) \text{ is } 
\rho_\epsilon(\G_L)\text{-stable} \iff  c(g)= \mu(1-\phi^{-1}(g)) \text{ for all } g\in \G_L.\]

Noting  the restriction map $\rH^1(\Q, \phi^{-1})\to \rH^1(\Q_p, \bar\Q_p)$ is injective, 
one deduces that: 
\begin{align}\label{crit-red}  
\rho_\epsilon  \text{ is } \G_{\Q} \text{-reducible} \iff  \rho_\epsilon  \text{  is nearly-ordinary} \iff [c]=0 
\overset{\eqref{exact-adjoint}}{\iff}  [\rho_\epsilon]\in \rH^1(\Q,W_\rho). 
\end{align}
However, as  $\phi(\tau)\neq1$, the $\G_{\Q}$-stable line is unique (when exists), while 
$\phi_{|\G_{\Q_p}}=\mathbf{1}$ implies that if there exists a $\G_{\Q_p}$-stable line then they all are. 
In particular $t^{\nord}_{\rho} \simeq \rH^1(\Q,{W_\rho})\oplus \bar\Q_p$.

To prove that $t^{\red}_{\rho}=t^{\ord}_{\rho}\simeq \rH^1(\Q,{W_\rho})$ it suffices  to show that
any nearly-ordinary $\rho_\epsilon$ (in particular any reducible $\rho_\epsilon$) is in fact ordinary, {\it i.e.} 
there exists a $\G_{\Q_p}$-stable line  $F_{\rho_\epsilon}= \bar\Q_p[\epsilon](e_1+ \epsilon \mu \cdot e_2)$
such that $\I_{\Q_p}$ acts trivially on $(\bar\Q_p[\epsilon])^2/F_{\rho_\epsilon}$. 
As  $c(\tau)=0$, the condition $[c]=0$ implies that $c=0$. It then follows  from   \eqref{Y-line} and \eqref{rho-epsilon} that 
\begin{align}\label{eq:ord-line}  
\rho_{\epsilon| \G_{\Q_p}}=
\left( \begin{matrix} 1 & 0\\ \epsilon \mu  & 1 \end{matrix}\right)
\left( \begin{matrix} 1+\epsilon (a+\mu\eta ) & \eta+ \epsilon(b +\eta a)\\  0 & 1+\epsilon (d-\mu \eta)  \end{matrix}\right)
\left( \begin{matrix} 1 & 0\\ -\epsilon \mu  & 1 \end{matrix}\right), 
\end{align} 
hence  $\I_{\Q_p}$ acts trivially on 
$(\bar\Q_p[\epsilon])^2/F_{\rho_\epsilon}$ if and only if $d=\mu \eta$ on $\I_{\Q_p}$. 
As  $d_{|\G_{\Q_p}}\in \rH^1(\Q_p, \bar\Q_p)$ and  $\eta_{|\I_{\Q_p}}$ is a basis of the image of the restriction map  $ \rH^1(\Q_p, \bar\Q_p)\to \rH^1(\I_{\Q_p}, \bar\Q_p)$, there exists a unique $\mu\in \bar\Q_p$ such that  $d_{|\I_{\Q_p}}=\mu\eta_{|\I_{\Q_p}}$. 

Since $W_\rho\simeq W^0_\rho\oplus \mathbf{1}$,  it is clear that $t^{\nord}_{\rho,0} \simeq \rH^1(\Q,{W^0_\rho})\oplus \bar\Q_p$ and  $t^{\red}_{\rho,0}=t^{\ord}_{\rho,0}\simeq \rH^1(\Q,{W^0_\rho})$.

By definition   $[\rho_\epsilon]\in t^{\eis}_{\rho}$ if and only if $d\in \rH^1(\Q,\bar\Q_p)$ is unramified, {\it i.e. } $d=0$. Hence
\begin{align}\label{eq:t-eis}
\begin{array}{lclcl}
t^{\eis}_{\rho}    & \simeq & \ker \left(\rH^1(\Q,W_\rho) \to \rH^1(\Q,\bar\Q_p)\right) \text{ and } \\
t^{\eis}_{\rho,0}    & \simeq &    \ker \left(\rH^1(\Q,W^0_\rho) \to \rH^1(\Q,\bar\Q_p)\right), 
\end{array}
\end{align}
 where the maps come from  the natural homomorphism $W_{\rho}\to \Hom_{\bar\Q_p}(V_{\rho}/F_{\rho}, V_{\rho}/F_{\rho})=\bar\Q_p$ sending $\left[\left(\begin{smallmatrix}a & b \\ 0 & d \end{smallmatrix}\right) \right]$ to $[d]$. 
As $\rH^0(\Q,\rho)=\{0\}$,  there are  exact  sequences in cohomology
\begin{align}
 \rH^0(\Q,W_\rho)  \xrightarrow{\sim} &  \rH^0(\Q,\bar\Q_p) \to  \rH^1(\Q,W'_\rho) \to \rH^1(\Q,W_\rho) \to  \rH^1(\Q,W_\rho/W'_\rho),   \nonumber \\
 & \rH^0(\Q,\bar\Q_p) \xrightarrow{\sim} \rH^1(\Q,\phi) \to \rH^1(\Q,W^0_\rho) \to \rH^1(\Q,W^0_\rho/W'_\rho{}^0).
\label{exact-W} 
\end{align}
Here we have used that $W'_\rho\simeq \rho$ and $W'_\rho{}^0\simeq \phi$, because 
 by  \eqref{rho-adjoint}  one has  $\rho \left(\begin{smallmatrix}a & b \\ 0 & 0 \end{smallmatrix}\right) \rho^{-1}=  \left(\begin{smallmatrix}a & b\phi-a\eta \\ 0 & 0 \end{smallmatrix}\right) $. It follows then  from \eqref{eq:t-eis} that  
$t^{\eis}_{\rho}=\rH^1(\Q,{W'_\rho})$ and   $t^{\eis}_{\rho,0}=\{0\}$.  
 \end{proof}

To determine the dimensions of these cohomology groups we will  need the following lemma.

\begin{lemma} One has $\dim_{\bar\Q_p} \rH^1(\Q, \rho)=1$ and $\dim_{\bar{\Q}_p} \rH^2(\Q,\rho)=0$.
\end{lemma}
\begin{proof} The global Euler characteristic formula yields: 
\[\dim \rH^2(\Q,\phi)=\dim \rH^1(\Q,\phi)-\dim \rH^0(\Q,\phi)+ \dim \rH^0(\R,\phi)-\dim(\phi)=1-0+0-1=0.\]
Since $W'_\rho\simeq \rho$ the exact sequence \eqref{exact-W} implies that  $\rH^1(\Q, \rho)\simeq \rH^1(\Q, \bar\Q_p)$ is 
$1$-dimensional. Another application of   Euler's global characteristic formula yields  $\rH^2(\Q,\rho)=\{0\}$.
\end{proof}
Since $W^0_\rho\simeq \rho \simeq W'_\rho$ it follows  then from Proposition~\ref{lemmatd} that 
\begin{align} \label{dim t_R}
\dim t^{\red}_{\rho}=\dim t^{\ord}_{\rho}=2,  \,\,
\dim t^{\eis}_{\rho}=\dim t^{\red}_{\rho,0}=\dim t^{\ord}_{\rho,0}=1 \text{  and }  t^{\eis}_{\rho,0}=\{0\}. 
\end{align}

\begin{rem}  By the proof of Proposition~\ref{lemmatd}, a deformation  
$[\rho_{\epsilon}]\in t^{\red}_{\rho}\setminus t^{\eis}_{\rho}$ is represented by a lift 
$\rho_{\epsilon}=\left( \begin{smallmatrix} a & b\\ 0 & d \end{smallmatrix} \right)$ 
such that   $d$ is  ramified at $p$, and yet $\rho_{\epsilon}$ admits a $\G_{\Q_p}$-filtration with 
unramified  quotient. The non-uniqueness of the $\G_{\Q_p}$-stable line is due to the fact that $\phi$ is trivial on $\G_{\Q_p}$. 
\end{rem}

\subsection{An application of  Baker--Brumer's Theorem} 
The fixed field $H$ of $\ker(\phi)\subset \G_\Q$ is a totally imaginary cyclic extension of $\Q$ of degree $2r\geqslant 2$, in which $p$ splits completely. 
The embedding $\iota_p: \bar{\Q} \hookrightarrow \bar\Q_p$ determines  a place  $v_0$ of $H$ and 
an  embedding $\G_{\Q_p}=\G_{H_{v_0}}\subset \G_H$ yielding a canonical restriction map 
\begin{align} \label{res-iota}
\res_p: \rH^1(H, \bar\Q_p) \to  \rH^1(\Q_p, \bar\Q_p).
\end{align}
Fixing a generator $\sigma$  of  $G=\Gal(H/\Q)$ allows us to number the places in $H$ above $p$
as $v_i=v_0 \circ \sigma^i$, $0\leqslant i \leqslant 2r-1$. 
 Let $\cO_H$ (resp. $\cO_{v_i}$) be the ring of integers of $H$ (resp. $H_{v_i}$). 

Recall the standard choice of $p$-adic logarithm  $\log_p$ sending $p$ to $0$. 
Denoting  $\ord_p:\Q_p^\times\to \Z$ the valuation, we consider the $\bar\Q_p$-linear maps 
\begin{align*}
\log_{v_0} : \cO_H[\tfrac{1}{p}]^\times\otimes \bar\Q_p&\longrightarrow \bar\Q_p & \ord_{v_0} : \cO_H[\tfrac{1}{p}]^\times\otimes \bar\Q_p& \longrightarrow \bar\Q_p\\
u \otimes x&\mapsto \log_p (\iota_p(u)) x &u \otimes x& \mapsto \ord_p (\iota_p(u)) x
\end{align*}

Given any   {\it odd} character $\psi$ of $G$, the $\psi^{-1}$-eigenspace of  $\cO_H[\tfrac{1}{p}]^\times\otimes \bar\Q_p$
is a line, and we let   $u_\psi$  be a basis. Note that 
$\ord_{v_0}(u_\psi)\neq 0$ since otherwise, by $\psi^{-1}$-equivariance,  one would have $\ord_{v_i}(u_\psi)=\ord_{v_0}(\sigma^{-i}(u_\psi))= 0$
for all $0\leqslant i \leqslant 2r-1$, which is impossible since the $\psi^{-1}$-eigenspace of  $\cO_H^\times\otimes \bar\Q_p$ is zero. 
 Following \cite[(7)]{DDP} we define the  $\cL$-invariant of  $\psi$ as 
 \begin{align}\label{eq:l-inv}
\cL(\psi):=-\frac{\log_{v_0}(u_{\psi})}{\ord_{v_0}(u_\psi)}\in \bar\Q_p.
\end{align}

As well known, $\rH^1(\Q, \psi)$ is a line isomorphic to  the $\psi^{-1}$-eigenspace of $\rH^1(H, \bar\Q_p)$. Fix  $[\eta_\psi]\in \rH^1(\Q, \psi)$ whose restriction  to $\I_{\Q_p}$ corresponds   to $\log_p$ (as for $\eta_1$ defined in \S\ref{def-eta}). Then 
\[ L_{\bar\Q} := \bar\Q \eta_1 \oplus \bigoplus_{ \psi \text{  odd } }  \bar\Q \eta_\psi \]
is a  $\bar\Q$-linear subspace of   $L_{\bar{\Q}} \otimes_{\bar{\Q}} \bar\Q_p=\rH^1(H, \bar\Q_p)$.

\begin{prop} \label{L-invariant}
The element  $\eta_\psi-\eta_{\mathbf{1}}$ is unramified at $p$ and $(\eta_\psi-\eta_{\mathbf{1}})(\Frob_p)=\cL(\psi^{-1})$.
\end{prop}

\begin{proof} There is an exact sequence of  $\bar\Q_p[G]$-modules  
\begin{align}
\label{restr hs}
0\to \mathrm{Hom}(\G_H, \bar\Q_p) \to  \bigoplus_{i=0}^{2r-1} \mathrm{Hom}(H_{v_i}^\times, \bar\Q_p)\to \mathrm{Hom}(\cO_H[\tfrac{1}{p}]^\times, \bar\Q_p). 
\end{align}
where $\xi: \G_H \to \bar\Q_p$  is sent to the collection of   maps $\xi_i: H_{v_i}^\times \to \bar\Q_p$, $0\leqslant i \leqslant 2r-1$, 
defined by taking the restriction  to $H_{v_i}^\times  \subset \widehat{ H_{v_i}^\times} \simeq \G_{H_{v_i}}^\mathrm{ab}$. 
Then  $ (\eta_\psi-\eta_\mathbf{1})(\Frob_p)=(\eta_{\psi,0}-\eta_{\mathbf{1},0})(\varpi_0)$, 
where $\varpi_0$ denotes a uniformiser of $H_{v_0}$. Denoting by  $e$  the exponent of the Hilbert class group of $H$, there exists
 $x_{0}\in \cO_H[\tfrac{1}{p}]^\times$ whose valuation at $v_0$ is $e$, while it is  $0$ at all other finite places of $H$. 
We can write  $x_0=\varpi_0^e y$ with $y \in \cO_{v_0}^\times$ and we have  
\[
(\eta_{\psi,0}-\eta_{\mathbf{1},0})(x_0)=(\eta_{\psi,0}-\eta_{\mathbf{1},0})(\varpi_0^e y)=e\cdot(\eta_{\psi,0}-\eta_{\mathbf{1},0})(\varpi_0)=
e\cdot(\eta_\psi-\eta_\mathbf{1})(\Frob_p).
\]
Since $x_0\in \cO_H[\tfrac{1}{p}]^\times $ and $\eta_\psi-\eta_\mathbf{1}\in \mathrm{Hom}(\G_H, \bar\Q_p)$ 
the exact  sequence \eqref{restr hs} implies that 
\begin{align}
\label{eta components}
(\eta_{\psi,0}-\eta_{\mathbf{1},0})(x_0)=-\sum_{i=1}^{2r-1} (\eta_{\psi,i}-\eta_{\mathbf{1},i})(x_0).
\end{align}
Since by definition  $\eta_\psi$  belongs to the $\psi^{-1}$-eigenspace for the $G$-action, it is entirely determined by $\eta_{\psi,0}$. 
More precisely one has 
$\eta_{\psi,i}=\psi(\sigma)^{-i} (\eta_{\psi,0} \circ \sigma^i)$ for all $0\leqslant i \leqslant 2r-1$. 
Combining this with \eqref{eta components}  and observing that $\sigma^i(x_0)\in \cO_{v_0}^\times$
 for every $1\leqslant i \leqslant 2r-1$, we obtain
 \begin{align}\label{eq:eta-psi}
 \begin{split}
&e\cdot(\eta_\psi-\eta_{\mathbf{1} })(\Frob_{v_0})=(\eta_{\psi,0}-\eta_{\mathbf{1},0})(x_0)=-\sum_{i=1}^{2r-1}\ (\eta_{\psi,i}-\eta_{\mathbf{1},i})(x_0)=\\
&=-\sum_{i=1}^{2r-1}(\psi(\sigma)^{-i} \eta_{\psi,0}  (\sigma^{i}(x_0))-\eta_{\mathbf{1},0} ( \sigma^{i}(x_0)))
=-\sum_{i=0}^{2r-1} (\psi(\sigma)^{-i}-1)\log_p  (\iota_p(\sigma^{i}(x_0) ))
\end{split}
\end{align}
because the restrictions of $\eta_\psi$ and of $\eta_{\mathbf{1}}$ to $\I_{H_{v_0}}=\I_{\Q_p}$ are given by $\log_p$. Observe first that 
\[\sum_{i=0}^{2r-1} \log_p  (\iota_p(\sigma^{i}(x_0) )= \log_p  (\iota_p(\mathrm{N}_{H/\Q}(x_0) )\in 
\log_p  (\iota_p(\pm p^\Z))=\{0\},\]
and
$\sum_{i=0}^{2r-1} \psi(\sigma)^{-i} \log_p  (\iota_p(\sigma^{i}(x_0)))=\log_{v_0}(u_{\psi^{-1}})$, where
$u_{\psi^{-1}}=\sum_{i=0}^{2r-1}  \sigma^{i}(x_0)\otimes \psi(\sigma)^{-i} $. As 
$u_{\psi^{-1}}$   belongs to the
 $\psi$-eigenspace of  $\cO_H[\tfrac{1}{p}]^\times\otimes \bar\Q_p$, and as $\ord_{v_0}(u_{\psi^{-1}})=\ord_{v_0}(x_0\otimes 1)= e$, one has 
  $\log_{v_0}(u_{\psi^{-1}})=-e\cdot \cL(\psi^{-1})$   by definition \eqref{eq:l-inv}. Combining this with \eqref{eq:eta-psi} yields the claim. 
\end{proof}

\begin{prop}  \label{independence log}\
\begin{enumerate}
\item \label{independence alpha psi}The  $\cL(\psi)$ are linearly independent over $\bar\Q$, as $\psi$ runs over all odd characters of $G$. 
 \item \label{injective bar Q}The  restriction to  $L_{\bar{\Q}}$ of the map  $\res_p$ defined in \eqref{res-iota}  is injective. 
 \end{enumerate}
 \end{prop}

\begin{proof}
\eqref{independence alpha psi}
 Suppose that $\sum_{\psi}m_{\psi} \cL(\psi)=0$ for some  $m_{\psi}\in \bar\Q$. As in the proof of Proposition~\ref{L-invariant}, we denote by $e$  the exponent of the Hilbert class group of $H$, and fix an element 
 $x_{0}\in \cO_H[\tfrac{1}{p}]^\times$ with valuation $e$ at $v_0$ and  $0$ at  all other finite places  of $H$. It follows that 
 \[
 \sum_{\psi \text{ odd}} m_{\psi} \sum_{i=0}^{2r-1}\ (\psi(\sigma)^i-1)\log_p  (\iota_p(\sigma^{i}(x_0) ))=0.
 \]
Since the $i=0$ summand vanishes, letting $m_\mathbf{1}=-\sum_{\psi \text{ odd}} m_{\psi}$ the formula can be written as 
\[
 \sum_{i=1}^{2r-1} \log_p  (\iota_p(\sigma^{i}(x_0) )) \left (\sum_{\psi \text{ odd or } \psi =\mathbf{1}} m_{\psi} \psi(\sigma)^i\right )=0
\]
We claim that the values $\{\log_p  (\iota_p(\sigma^{i}(x_0) ))\}_{1\leqslant i \leqslant 2r-1}$ are linearly independent over $\Q$. To see this suppose 
$\log_p(\iota_p(x))=0$ for some  element $x=\prod_{1\leqslant i \leqslant 2r-1} \sigma^i(x_0)^{n_i}$ with $n_i\in \Z$. 
As  $\iota_p(x)\in \bar\Z_p^\times$  this implies that  $x$ is a root of unity in $H$,  leading to $n_i=\ord_{v_i}(x)=0$ for all $1\leqslant i \leqslant 2r-1$. 
 By  Baker--Brumer's Theorem~\cite{brumer}, the elements $\{\log_p  (\iota_p(\sigma^{i}(x_0) ))\}_{1\leqslant i \leqslant 2r-1}$ are therefore  linearly   independent over $\bar\Q$, leading to  
 \begin{align}\label{odd char}
\sum_{\psi \text { odd or } \psi=1}m_\psi \psi(\sigma)^i   =0
\end{align}
for any $1\leqslant i \leqslant 2r-1$. Moreover, as $m_\mathbf{1}=-\sum_{\psi \text{ odd}} m_{\psi}$,  \eqref{odd char}   holds for $i=0$ as well.

 Let $\psi_1, \psi_2, \dots \psi_{r}$ be a numbering of the odd characters of $G$. The condition \eqref{odd char} can be rewritten as $(m_\mathbf{1}, m_{\psi_1}, \dots, m_{\psi_{r}})\cdot M=(0, 0, \dots,0)$, where 
 \[
M=\left (\begin{matrix}
1&1 & 1&\dots &1\\
1 & \psi_1(\sigma) & \psi_1(\sigma)^2 & \dots & \psi_1(\sigma)^{2r-1}  \\
\dots & \dots &\dots & \dots & \dots \\
1 & \psi_r(\sigma)& \psi_r(\sigma)^2 & \dots & \psi_r(\sigma)^{2r-1}  
\end{matrix}
\right )
\]
As  $2r \geqslant r+1$,  $M$ contains as a sub-matrix the Vandermonde matrix  of 
$( 1 , \psi_1(\sigma),  \dots, \psi_{r}(\sigma))$ which as well-known is invertible, implying that 
$m_\psi=0$ for every $\psi$.

\eqref{injective bar Q} It suffices to notice that the kernel of the restriction map $L_{\bar\Q_p} \to \rH^1(\I_{\Q_p}, \bar\Q_p)$ is spanned by $\{(\eta_{\mathbf{1}}-\eta_\psi)\}_{\psi \text{ odd}}$. Combining  Proposition~\ref{L-invariant} with \eqref{independence alpha psi} yields
  the desired result. 
 \end{proof}

\subsection{The  tangent space for cuspidal deformations}

Let $t_{\cusp}=\cD_{\cusp}(\bar\Q_p[\epsilon])$ and $t_{\cusp}^0=\cD_{\cusp}^0(\bar\Q_p[\epsilon])$ be the tangent space and the relative tangent space to the functor $\cD_{\cusp}$.

\begin{prop}\label{cuspidal tangent}
We have  $\dim_{\bar\Q_p} t_{\cusp}=1$  and  $\dim_{\bar\Q_p} t_{\cusp}^0=0$. 
\end{prop}

\begin{proof} By definition 
\[t_{\cusp}=\cD_{\cusp}(\bar\Q_p[\epsilon])=\left\{(\rho_\epsilon, \rho'_\epsilon) \in t^{\ord}_{\rho} \times t^{\ord}_{\rho'} \mid
\tr(\rho_\epsilon)= \tr(\rho'_\epsilon)\text{ and } \chi_\epsilon(\Frob_p)=\chi'_\epsilon(\Frob_p)\right\}.\]
 By the proof of Proposition~\ref{lemmatd} an element $[\rho_\epsilon]\in t^{\ord}_{\rho}\simeq t^{\red}_{\rho}\simeq\rH^1(\Q, W_\rho)$ can be written as: 
\[ \rho_\epsilon =\left(1+\epsilon 
\begin{pmatrix} a & b\\ 0 & d \end{pmatrix} \right )\rho=\begin{pmatrix}
\phi(1+\epsilon a) & \eta (1+\epsilon a)+b \epsilon\\ 0 & 1+\epsilon d \end{pmatrix}  \]
for $ \left [ \begin{smallmatrix} a & b\\ 0 & d  \end{smallmatrix} \right ] \in \rH^1(\Q, W_\rho). 
$
In particular $a, d \in \rH^1(\Q, \bar\Q_p)$. Recall the generator $\eta_\mathbf{1}$ of $\rH^1(\Q, \bar\Q_p)$ whose restriction at $\I_{\Q_p}$ is $\log_p$. 
Writing
$a=\lambda \eta_\mathbf{1}$,  $ d=\mu  \eta_ \mathbf{1} $ with $\lambda, \mu \in \bar\Q_p$ yields: 
 \begin{align}\label{cusp-trace}
\tr(\rho_\epsilon)=1+\phi+\epsilon( \lambda \phi+ \mu) \eta_\mathbf{1}
\text{ and }  \det(\rho_\epsilon)=\phi(1+\epsilon( \lambda + \mu) \eta_\mathbf{1}). 
\end{align}
By \eqref{eq:ord-line},  the ordinary filtration $F_{\rho_\epsilon}$ has basis $e_1+\epsilon\mu e_2$ and 
$\rho_\epsilon(\G_{\Q_p})$ acts on the quotient by the character $ \chi_\epsilon=1+\epsilon(d-\mu \eta)$. 
  It follows from  Proposition~\ref{L-invariant} that 
 \begin{align}\label{U_p equation1}  
 \chi_\epsilon(\Frob_p)=1+\epsilon \mu (\eta_\mathbf{1}-\eta)(\Frob_p)=1-\mu \cL(\phi^{-1}) \epsilon.
\end{align}

Since $\rho'=\left( \begin{smallmatrix} \mathbf{1} & \phi \eta'\\ 0 & \phi  \end{smallmatrix}\right)
=\left( \begin{smallmatrix} \phi^{-1} &  \eta'\\ 0 & \mathbf{1}   \end{smallmatrix}\right)\otimes \phi$ one can  describe  
$t^{\ord}_{\rho'}\simeq t^{\ord}_{\rho'\otimes\phi^{-1} }$ by simply replacing $\phi$ by $\phi^{-1}$ and $\eta$ by $\eta'$ in 
the above description of $t^{\ord}_{\rho}$. One then finds that: 
  \begin{align}\label{U_p equation2}
\begin{split}
\tr(\rho'_\epsilon)=\phi(1+\phi^{-1}+\epsilon( \lambda' \phi^{-1}+ \mu') \eta_\mathbf{1})=
1+\phi+ \epsilon(\lambda' + \mu'\phi) \eta_\mathbf{1}, \\  
 \chi'_\epsilon(\Frob_p)
=1+\epsilon \mu' (\eta_\mathbf{1}-\eta')(\Frob_p)=1- \mu' \cL(\phi) \epsilon. 
\end{split}
\end{align}
From \eqref{cusp-trace}, \eqref{U_p equation1} and \eqref{U_p equation2} one sees that 
  \begin{align} \label{system}
  (\rho_\epsilon, \rho'_\epsilon)\in t_{\cusp}\iff \lambda=\mu',\, \mu=\lambda' \text{  and }  \mu \cL(\phi^{-1}) =\lambda \cL(\phi).
\end{align}
By Proposition~\ref{independence log},
$\cL(\phi) $ and $\cL(\phi^{-1})$ are both non-zero, hence $\dim t_{\cusp}=1$. 

To compute the relative tangent space $t_{\cusp}^0$ it suffices to add to  \eqref{system} the condition $\det \rho_\epsilon=\phi$, which is equivalent to 
$\lambda+\mu=0$. By Proposition~\ref{independence log}, $\cL(\phi) $ and $\cL(\phi^{-1})$ are linearly independent over $\bar\Q$ if 
 $\phi$ is not quadratic, while when $\phi$ quadratic one has $\cL(\phi)=\cL(\phi^{-1}) \neq 0$. 
In either case  $\cL(\phi)  +\cL(\phi^{-1})\neq 0$, hence   the equation $\lambda+\mu=0$ is linearly independent from  \eqref{system}, and  $\dim_{\bar\Q_p}t_{\cusp}^0=0$.
 \end{proof}

\begin{cor}
We have $t^{\ord}_{\rho}=t_{\cusp} \oplus t^{\eis}_{\rho}$. 
\end{cor}
\begin{proof} As in the proof of Proposition~\ref{cuspidal tangent} one can use 
$\lambda$ and $\mu$ as coordinates on $t^{\ord}_{\rho}$ and by \eqref{system} the equation defining $t_{\cusp}$ is  $\mu \cL(\phi^{-1}) =\lambda \cL(\phi)$. On the other hand  by  \eqref{eq:t-eis} the equation defining $t^{\eis}$ is  $\mu=0$, and we 
 have seen that $\cL(\phi)\neq 0$.  
\end{proof}

\subsection{Iwasawa cohomology}\label{iwasawa-coh}
Let $\varepsilon_p:\G_\Q  \twoheadrightarrow \Gal(\Q(\mu_{p^{\infty}})/\Q) = \Z_p^\times$ be the $p$-adic 
 cyclotomic character and let  $\omega_p:\G_\Q\twoheadrightarrow \Gal(\Q(\mu_{2p})/\Q) = (\Z/2p\Z)^\times\to  \Z_p^\times$ be the $p$-adic  Teichm\"uller character. Let $\nu=2$ if $p=2$ and $\nu=1$ otherwise.  
 The cyclotomic $\Z_p$-extension  $\Q_\infty$ of $\Q$ is the fixed field of 
 \[\varepsilon_p\omega_p^{-1}:\G_\Q\to 1+p^\nu\Z_p.\] 
The universal cyclotomic character   $\chi_{p}: \G_\Q\to \varLambda^\times$ is obtained by composing 
$\varepsilon_p\omega_p^{-1}$ with 
\begin{align}\label{eq:univ-cyc}
1+p^\nu\Z_p \to \Z_p\lsem 1+p^\nu\Z_p\rsem^\times \xrightarrow{\sim}\Z_p\lsem X \rsem^\times\hookrightarrow \varLambda^\times,
 \end{align}
where the isomorphism in the middle sends  $1+p^\nu$ to $1+X$. 
 One has: 
\begin {align} \label{eq:univ-cyc-modX2}
 \chi_{p} \equiv1 - \frac{\eta_\mathbf{1}}{\log_p(1+p^\nu)}X\pmod{X^2}  , 
 \end{align}
where  $\eta_\mathbf{1}:\G_\Q\to \bar\Q_p$ is the cyclotomic homomorphism defined in \S\ref{def-eta}
sending $ \Frob_\ell$ to $-\log_p(\ell)$ for all $\ell\neq p$. It follows that 
 \begin{align}\label{derivative-chi-cyc}
 \left.\tfrac{d}{dX}\right|_{X=0}\chi_{p}= - \tfrac{\eta_\mathbf{1}}{\log_p(1+p^\nu)}.
 \end{align}

Using Class Field Theory  one can show that 
 $\Phi=\phi \chi_{p}$ is the  deformation of $\phi$ to its universal deformation ring $\varLambda$ (see \cite[\S6]{bellaiche-dimitrov}).  For
 $n\in \Z_{\geqslant 1}$,  we let $\Phi_n=\Phi\mod{X^n}$.

We will prove the existence of a non-torsion cohomology class in the Iwasawa cohomology group $\rH^1(\G_{\Q}^{Np},\Phi)$, where $\G_{\Q}^{Np}$ is the Galois group of the maximal extension of $\Q$ unramified outside  $Np$ and $\infty$.  This will be  used in \S\ref{sec:Eisenstein} to show that there exists a surjection $\cR^{\eis}_\rho \twoheadrightarrow \varLambda$. 
 
\begin{prop} \label{def-red-rho}
\begin{enumerate}
\item 
One has $\rH^2(\G_{\Q}^{Np}, \Phi^{\pm 1})=0$. Moreover $\rH^1(\G_{\Q}^{Np}, \Phi^{\pm 1})$ is a free $\varLambda$-module of rank $1$ and  
for all $n\in \Z_{\geqslant 1}$   the natural homomorphism is an isomorphism: 
\[ \rH^1(\G_{\Q}^{Np}, \Phi^{\pm 1})\otimes_\varLambda \varLambda/(X^n)
\xrightarrow{\sim} \rH^1(\G_{\Q}^{Np}, \Phi_n^{\pm 1}).\]
\item  For every $n\geqslant 1$,  the natural restriction map is injective:
\[ \rH^1(\G_{\Q}^{Np}, \Phi_n^{\pm 1}) \to \rH^1(\Q_p, \Phi_n^{\pm 1}).\]
\end{enumerate}
\end{prop}

\begin{proof}
(i) The short exact sequence of $\bar\Q_p[\G_{\Q}^{Np}]$-modules
 \[
0 \to \Phi_{n-1}^{\pm 1} \xrightarrow{\cdot  X}  \Phi_n^{\pm 1} \to \phi^{\pm 1} \to 0 
\]
yields  for $i\geqslant 1$ a long exact sequence in cohomology 
\begin{align}\label{eq:les-Phi-n}
 \rH^{i-1}(\G_{\Q}^{Np}, \phi^{\pm 1})
 \to \rH^i(\G_{\Q}^{Np}, \Phi_{n-1}^{\pm 1}) \xrightarrow{\cdot  X}  \rH^i(\G_{\Q}^{Np}, \Phi_n^{\pm 1}) \to  \rH^i(\G_{\Q}^{Np}, \phi^{\pm 1}).
\end{align}
As $\rH^2(\G_{\Q}^{Np},\phi^{\pm 1})=0$  by the global Euler characteristic formula, one
obtains   $\rH^2(\G_{\Q}^{Np}, \Phi_n^{\pm 1})=\{0\}$ by induction on $n$  (noting that $\Phi_1=\phi$).  
Similarly, as $\dim\rH^1(\Q,\phi^{\pm 1})=1$ and $\rH^0(\Q,\phi^{\pm 1})=\{0\}$, one deduces that  $\dim_{\bar\Q_p} \rH^1(\G_{\Q}^{Np}, \Phi_n^{\pm 1})=n$ for all $n\in \Z_{\geqslant 1}$. 

The short exact sequence of $\varLambda[\G_{\Q}]$-modules
\begin{align*}
0 \to \Phi^{\pm 1} \xrightarrow{\cdot X^n} \Phi^{\pm 1}\to \Phi_n^{\pm 1} \to 0 
\end{align*}
yields a long  exact sequence of $\varLambda$-modules in cohomology  (see \cite[Proposition~3.5.1.3]{nekovar-book}): 
\begin{align}\label{les-Lambda}
0 \to \rH^1(\G_{\Q}^{Np}, \Phi^{\pm 1}) \xrightarrow{\cdot X^n} \rH^1(\G_{\Q}^{Np}, \Phi^{\pm 1}) \to  \rH^1(\G_{\Q}^{Np}, \Phi_n^{\pm 1}) 
 \to \rH^2(\G_{\Q}^{Np}, \Phi^{\pm 1}) \xrightarrow{\cdot X^n} \rH^2(\G_{\Q}^{Np}, \Phi^{\pm 1}) \to 0
\end{align}
According to \cite[Proposition~4.2.3]{nekovar-book},  $\rH^i(\G_{\Q}^{Np}, \Phi^{\pm 1})$ is a  $\varLambda$-module of finite type for $i \in \{0,1,2\}$. Therefore, Nakayama's lemma applied to \eqref{les-Lambda} for $n=1$
implies that $\rH^2(\G_{\Q}^{Np}, \Phi^{\pm 1})=\{0\}$ while $\rH^1(\G_{\Q}^{Np}, \Phi^{\pm 1})$ is a cyclic $\varLambda$-module. 
Moreover \eqref{les-Lambda} for  an arbitrary $n$ yields 
\[\rH^1(\G_{\Q}^{Np}, \Phi^{\pm 1})\otimes_\varLambda \varLambda/(X^n) \simeq \rH^1(\G_{\Q}^{Np}, \Phi_n^{\pm 1})\] 
which has dimension $n$. Hence  $\rH^1(\G_{\Q}^{Np}, \Phi^{\pm 1})$ is a free $\varLambda$-module of rank $1$.

(ii) 
Using (i), the long exact sequence  \eqref{eq:les-Phi-n} yields a commutative diagram with exact rows: 
\begin{align}\label{eq:les-Phi-nloc} 
\xymatrix{ 
0 \ar[r] &   \rH^1(\G_{\Q}^{Np}, \Phi_{n-1}^{\pm 1}) \ar^{\cdot  X}[r] \ar_{\res_{n-1}}[d] &
\rH^1(\G_{\Q}^{Np}, \Phi_n^{\pm 1})  \ar@{->>}[r]\ar_{\res_n}[d]  & \rH^1(\G_{\Q}^{Np}, \phi^{\pm 1}) \ar_{\res}[d] \\
\rH^{0}(\G_{\Q_p}, \mathbf{1}) \ar^{\delta_{n-1}}[r]  & \rH^1(\G_{\Q_p}, \Phi_{n-1}^{\pm 1}) \ar^{\cdot  X}[r]   & \rH^1(\G_{\Q_p}, \Phi_n^{\pm 1})  \ar@{->>}[r] & \rH^1(\G_{\Q_p}, \mathbf{1})}
\end{align}
where the vertical arrows are the restriction maps. 
 Note that   the image of  the connecting homomorphism $\delta_{n-1}$
 is generated by the cohomology class of the cocycle $g \mapsto \frac{\chi_{p}^{\pm 1}(g)-1}{X}\pmod{X^{n-1}}$ which belongs to the $X$-torsion, because  $g \mapsto \chi_{p}^{\pm 1}(g)-1$ is a coboundary. 

We argue by induction on $n$. The  injectivity of  $\res=\res_1$  follows from Class Field Theory.  
Suppose that $\res_{n-1}$ is injective for some $n\geqslant 2$. It suffices then to show that
 \[ \mathrm{Im}(\res_{n-1}) \cap \mathrm{Im}(\delta_{n-1})=\{0\}. \] 
Let $[\eta_{\Phi^{\pm 1}}]$ be a generator of $\rH^1(\G_{\Q}^{Np}, \Phi^{\pm 1})$ as a free rank one $\varLambda$ module, and let  $[\eta_{\Phi^{\pm 1}_n}]$ denote its image in $\rH^1(\G_{\Q}^{Np}, \Phi_n^{\pm 1})$. 
Since  $\mathrm{Im}(\delta_{n-1})$ is $X$-torsion and $\res_{n-1}$ is injective,  an element of the above intersection is a scalar multiple of  $\res_{n-1}(X^{n-2} [\eta_{\Phi^{\pm 1}_{n-1}}])$. Moreover, as any non-trivial element of $\mathrm{Im}(\delta_{n-1})$ remains 
non-trivial when letting $X=0$, the above intersection automatically vanishes for $n\geqslant 3$. Finally for $n=2$, in virtue of   \eqref{eq:univ-cyc-modX2} one has to show that 
 $\res([\eta_{\phi^{\pm 1}}])$ and $\res(\eta_\mathbf{1}) $
generate distinct lines in $ \rH^1(\G_{\Q_p}, \mathbf{1})$ which follows from the fact that 
they have the same restriction to the inertia group, while taking different values on the Frobenius by 
Proposition~\ref{L-invariant}.  \qedhere
 \end{proof}

\section{\texorpdfstring{$p$}{}-adic families containing \texorpdfstring{$f$}{}}\label{sec:eigencurve}

In this section, we show that the completed strict local ring $\cT_{\cusp}$ of the cuspidal eigencurve $\cC_{\cusp}$ at  $f$ is isomorphic to the deformation ring $\cR_{\cusp}$ and conclude that the cuspidal eigencurve is \'etale over the weight space at $f$.

\subsection{Some basic facts on  the eigencurve}\label{basic-eigencurve}
Let $X/ \Z_p$ be the proper smooth modular curve of level (see \cite[Chap.~IV]{deligne-rapoport})
\[\Gamma=\begin{cases} \Gamma_1(N), &  \text{ if } N\geqslant 4,\\ 
\Gamma(3), &  \text{ if } N=3. \end{cases} \] 
 Let  $E \to X$ be the generalised elliptic curve endowed with the identity section $e:X \to E$, $\omega=e^*(\Omega_{E/X})$ be the conormal sheaf and $X^{\rig}$ be the rigid analytic space attached to the generic fibre of $X$ (note that by properness $X^{\rig}(\bar\Q_p)=X(\bar\Q_p)$). The analytification of $\omega$  is an invertible sheaf on $X^{\rig}$ and will be denoted  again by $\omega$.
 
 For $v\in \Q_{\geqslant 0}$ let $X(v)$ denote the open locus of $X^{\rig}$ where the truncated valuation of the Hasse invariant is at most $v$ (see \cite[\S3.1]{pilloni}); in particular  $X(0)$  is the ordinary locus. 
 We recall that  the weight space $\cW$  is  the rigid analytic space over $\Q_p$ such that  
\[\cW(\C_p)=\Hom_{\mathrm{cont}}(\Z_p^{\times}\times (\Z /N \Z)^{\times}, \C_p^\times).\]

For $\cU$  a connected open admissible affinoid of $\cW$ we let 
\[ \kappa_\cU: \Z_p^{\times}\times (\Z/N\Z)^\times \to \cO(\cU)^{\times} \]
 be the universal character, 
$\omega_\cU$ be the invertible sheaf on $X(v) \times \cU$  constructed in \cite[\S5.1]{pilloni}  under the assumption that either $v=0$,  or  that both $v>0$ and  $\cU$ are  sufficiently small.   
 By construction, for any  weight $k\in\Z_{\geqslant 1} \cap \cU $, 
 the sheaf $\omega_\cU$ specialises to  the invertible sheaf  $\omega^{\otimes k}$ on $X(v)$ (see \cite[Proposition~3.3]{pilloni}).

Consider the invertible sheaf $\omega_{\cU}(-D_\cU)$, where $D_\cU$ denotes the cuspidal divisor of $X(v)\times \cU$.
Note that $D_\cU$ does not depend on $v$ as the cusps of  $X$ all belong to  $X(0)$.
 For $N \geqslant 4$,  the space of families of overconvergent forms, resp. cuspforms, having  slope at most $s\in \Q_{\geqslant 0}$ is defined as: 
\[ M^{\dag,\leqslant s}_\cU= \varinjlim_{v>0} \rH^0(X(v) \times \cU,\omega_\cU)^{\leqslant s}\text{, resp.   }  
S^{\dag,\leqslant s}_\cU= \varinjlim_{v>0} \rH^0(X(v) \times \cU,\omega_{\cU}(-D_\cU))^{\leqslant s}.\]
The  space $M^{\dag,\leqslant s}_\cU$, resp.  $S^{\dag,\leqslant s}_\cU$,  is a locally free module of finite type   over the Banach algebra $\cO(\cU)$, and is contained  in the  space of families of $p$-adic forms, resp. cuspforms:
\[M_\cU=\rH^0( X(0) \times \cU,\omega_\cU)\text{, resp.   }   S_\cU=\rH^0( X(0) \times \cU,\omega_{\cU}(-D_\cU)).\]
For   $N=3$  we  define $M_\cU$ as the $(\Gamma_1(3)/\Gamma)$-invariants in $\rH^0( X(0) \times \cU,\omega_\cU)$ 
and we  similarly define $S_\cU$,  $M^{\dag,\leqslant s}_\cU$ and $S^{\dag,\leqslant s}_\cU$. 
We let $M^{\ord}_\cU= M^{\dag,\leqslant 0}_\cU$ denote the space of ordinary  families and by 
$S^{\ord}_\cU= S^{\dag,\leqslant 0}_\cU$ its cuspidal subspace. The notation is justified by the fact that any  $p$-adic form of slope $0$ is necessarily overconvergent (see  \cite[Proposition~6.2]{pilloni}).  

By construction  the eigencurve $\cC$ (resp. the cuspidal eigencurve $\cC_{\cusp}$) is a rigid analytic space  over $\Q_p$ admissibly covered  by  the affinoids attached to the  $\cO(\cU)$-algebras generated by the Hecke operators $T_\ell, \ell\nmid Np$ and $U_p$ acting on 
$\End_{\cO(\cU)}(M^{\dag,\leqslant s}_\cU)$ (resp. $\End_{\cO(\cU)}(S^{\dag,\leqslant s}_\cU)$), where both $s\in \Q_{\geqslant 0}$  and the open admissible affinoid $\cU\subset \cW$  vary. 
Thus, we obtain a closed immersion $\cC_{\cusp} \hookrightarrow \cC$ of  rigid curves and $\cC$ is endowed with a weight map $\kappa: \cC \to \cW$ which is  locally finite and flat. Moreover $\cC$  is reduced as the Hecke operators  act semi-simply on the generalised  eigenspace of any classical eigenform  which is regular at $p$ and has non-critical slope (see  \cite[Proposition~3.9]{chenevier-dmj}).  Similarly, one can obtain $\cC_{\rho'}$ (resp. $\cC_\rho$) by considering the $\cO(\cU)$-banach Hecke modules generated by $S^{\dag,\leqslant s}_\cU$ and the Eisenstein family $\cE_{\phi,\mathbf{1}}$ (resp. $\cE_{\mathbf{1},\phi}$).

There exists a ring homomorphism $\Z[T_{\ell},U_p]_{\ell\nmid Np} \to \cO_{\cC}(\cC)$ allowing one to  see $T_{\ell}$ and $U_p$ as    analytic functions on $\cC$ bounded by $1$. For any $k\in\cW$, the   points of $\kappa^{-1}(k)\subset \cC$ are in 
  bijection with the set of  systems of eigenvalues for   $\{T_\ell,U_p ; \ell \nmid Np\}$  acting on the space of  finite slope overconvergent  eigenforms of weight $k$ and tame level dividing $N$.

The locus of $\cC$ (resp. $\cC_{\cusp}$) where $|U_p|_p=1$ is open and closed in $\cC$ (resp. $\cC_{\cusp}$), and is called the ordinary locus. As the classical weights are Zariski dense in $\cW$, it follows from 
the classicality criterion for overconvergent forms (see \cite{coleman-ocmf}) 
that the classical points are Zariski dense in $\cC$. Moreover it follows from \cite[Corollary~2.6]{bellaiche-Inv}
that $\cC_{\cusp}$ has a Zariski dense set of points corresponding to classical cuspforms of weight at least $2$.

Since $\G_{\Q}$ is compact and  $\cO_{\cC}(\cC)$ is a reduced ring, there exists  (see \cite[\S7]{chenevier}) a unique continuous two-dimensional pseudo-character 
 \begin{align}\label{pseudo}
 \G_{\Q} \to \cO_{\cC}(\cC)
 \end{align}
 sending  $\Frob_\ell$ to $T_\ell$ for all $\ell \nmid Np$, and whose specialisation at  any  classical  point $g\in\cC(\bar\Q_p)$
 equals the  trace of the  Galois representation $\rho_g:\G_\Q \to \GL_2(\bar\Q_p)$ 
constructed by  Eichler-Shimura, Deligne and Deligne-Serre.

Let $X_{\mathrm{Iw}}/\Z_p$ be the proper flat semi-stable modular curve of level  $ \Gamma \cap \Gamma_0(p)$
 endowed with a  canonical   morphism $X_{\mathrm{Iw}} \to X$ obtained by forgetting the Iwahori $\Gamma_0(p)$-level structure
(see \cite[pp.101, 144]{deligne-rapoport}).  The theory of the canonical subgroup yields, for   $v> 0$ sufficiently small, a section 
of the latter morphism
 \begin{align}\label{can-section}
X(v)\to X^\times_{\mathrm{Iw}}(v),
\end{align}
where $X^\times_{\mathrm{Iw}}(v)$ is a neighbourhood of the connected (multiplicative) component  of the  ordinary locus in 
$X^{\rig}_{\mathrm{Iw}}$, containing the cusp $\infty$.  
 Given a classical modular form of level $\Gamma\cap \Gamma_0(p)$, 
  the pullback along this section of its restriction to  $X^\times_{\mathrm{Iw}}(v)$  yields an
overconvergent modular form having the same $q$-expansion at $\infty$. In particular,  classical eigenforms  of level 
$\Gamma_1(N)\cap \Gamma_0(p)$ give rise to points in $\cC$. Furthermore, if the classical  eigenform
vanishes at  all cusps of $X^\times_{\mathrm{Iw}}(v)$, {\it i.e.},  at all cusps lying in the $\Gamma_0(p)$-orbit of $\infty$, then the 
corresponding point belongs to  $\cC_{\cusp}$.

\subsection{Evaluation of ordinary families at cusps}\label{Ressection}
The connected  components of the cuspidal divisor $D_\cU$  of $ X(0) \times \cU$
are indexed by the finite  set $\Gamma\backslash \mathbb{P}^1(\Q)$ thus  is   fibered  over 
 $\Gamma_1(N)\backslash \mathbb{P}^1(\Q)$.

\begin{prop}\label{residuemapfamily}
Evaluation at  the cusps gives the following exact sequence  of $\cO(\cU)$-modules:
\begin{align}\label{eval-at-cups}
0 \to   S_\cU \to  M_\cU \xrightarrow{\res_\cU} \prod_{[\delta]\in \Gamma_1(N)\backslash \mathbb{P}^1(\Q)} \cO(\cU) \to 0.
\end{align}
\end{prop}

\begin{proof}We have an exact sequence of sheaves on $ X(0) \times \cU$:
\begin{align}\label{sheafexact}
0 \to \omega_{\cU}(-D_\cU) \to \omega_\cU \to \omega_\cU/\omega_{\cU}(-D_\cU) \to 0
\end{align}
where  the support of quotient sheaf $\omega_\cU/\omega_{\cU}(-D_\cU)$ is $D_\cU$.  Hence 
\[\rH^0( X(0) \times \cU,\omega_\cU/\omega_{\cU}(-D_\cU))=\rH^0(D_\cU,\omega_\cU/\omega_{\cU}(-D_\cU))
=\rH^0(D_\cU,\omega_\cU) =\prod_{[\delta]\in \Gamma\backslash \mathbb{P}^1(\Q)} \cO(\cU).\]

Since $ X(0) \times \cU$ is an affinoid and $\omega_{\cU}(-D_\cU)$ is a coherent sheaf (even invertible), one has
 \[\rH^1( X(0) \times \cU,\omega_{\cU}(-D_\cU))=0.\] 
 Applying the   functor global sections $\rH^0( X(0) \times \cU,\relbar) $ to \eqref{sheafexact}, 
 and further taking $(\Gamma_1(N)/\Gamma)$-invariants,   yields the desired result. 
\end{proof}

We recall that $M^{\ord}_\cU= e^{\ord}(M_\cU)$ and $S^{\ord}_\cU= e^{\ord}(S_\cU)$, where 
$e^{\ord}=\lim\limits_{n\to\infty}U_p^{n!}$ denotes Hida's  ordinary idempotent. Applying this idempotent to 
\eqref{eval-at-cups} yields the following result.

\begin{cor}\label{residuemorphismord} There exists a direct factor $C_\cU$ of $\displaystyle \prod_{[\delta]\in \Gamma_1(N)\backslash \mathbb{P}^1(\Q)} \cO(\cU)$  and 
an exact sequence  of $\cO(\cU)$-modules:
\begin{align}\label{eval-at-cups2}
0 \to   S^{\ord}_\cU \to  M^{\ord}_\cU \xrightarrow{\res_\cU} C_\cU \to 0.
\end{align}
\end{cor}

\begin{rem} Under the assumption that $p \geqslant 5$, Ohta gave in \cite{ohta-ES,ohta-eis} a  different description of  the residue map on the ordinary part  and proved its surjectivity. 
\end{rem}

\subsection{Construction of an irreducible deformation of \texorpdfstring{$\rho$}{}}
Consider the $p$-stabilised Eisenstein series $f$ as in the introduction. 
By   \cite[Proposition~1.3]{DLR1} (or Proposition~\ref{q-expa family} further below) 
the weight $1$ eigenform $f$ vanishes at  all cusps of $X_{\mathrm{Iw}}$  lying in the $\Gamma_0(p)$-orbit of $\infty$, hence 
as explained  above defines a point $f$ in $\cC_{\cusp}$, which lies in  the ordinary locus as $U_p(f)=1$. 
Recall that $\cT$ (resp. $\cT_{\cusp}$) denotes the 
 completed strict local ring   of $\cC$ (resp. $\cC_{\cusp}$) at $f$ and $\gm_\cT$ (resp. $\gm_{\cT_{\cusp}}$)  its maximal ideal. 

\begin{prop} \label{rhott} Let  $\chi_{\cT_{\cusp}}: \G_{\Q_p} \to \cT_{\cusp}^\times$ be 
 the unramified character sending  $\Frob_p$ to $U_p$.
\begin{enumerate}
\item There exists an irreducible  deformation $\rho_{\cT_{\cusp}}: \G_{\Q} \to \GL_2(\cT_{\cusp})$ of $\rho$ such that $\tr(\rho_{\cT_{\cusp}})(\Frob_\ell) = T_\ell$ for all $\ell \nmid Np$, and $\rho_{\cT_{\cusp} \mid \G_{\Q_p}}=\left(\begin{smallmatrix} * & * \\ 0 & \chi_{\cT_{\cusp}} \end{smallmatrix}\right)$.

\item There exists an irreducible deformation $\rho_{\cT_{\cusp}}':\G_{\Q} \to \GL_2(\cT_{\cusp})$ of $\rho'$ such that 
$\det(\rho_{\cT_{\cusp}}')= \det(\rho_{\cT_{\cusp}})$, $\tr( \rho_{\cT_{\cusp}}')= \tr( \rho_{\cT_{\cusp}})$
and $\rho'_{\cT_{\cusp} \mid \G_{\Q_p}}=\left(\begin{smallmatrix} * & * \\ 0 &  \chi_{\cT_{\cusp}} \end{smallmatrix}\right)$.
\end{enumerate}
\end{prop}

\begin{proof}  Since $\cT_{\cusp}$ is reduced, it injects in its total quotient field $L$.
By  \cite[Theorem~4]{wiles-ord}  the $2$-dimensional pseudo-character $\G_{\Q} \to \cO_{\cC}(\cC) \to \cT_{\cusp}\to  L$  lifting $\phi+\mathbf{1}$
(see \eqref{pseudo}) gives rise to a $p$-ordinary representation $\rho_{L}: \G_{\Q} \to \GL_2(L)$ such that $\rho_{L}(\tau)=\left(\begin{smallmatrix} -1 & 0 \\ 0 &  1 \end{smallmatrix}\right)$ and  $\tr (\rho_{L})(\Frob_\ell)= T_\ell$ for $\ell \nmid Np$. 
Moreover, each component of $\rho_{L}$ is absolutely irreducible, as the classical cuspidal points of $\cC_{\cusp}$ form a Zariski dense subset (see \S\ref{basic-eigencurve}). As $\dim_{\bar\Q_p}\rH^1(\Q,\phi)=1$,  \cite[Corollary~2]{bellaiche-chenevier} implies that  there exists a conjugate of $\rho_L: \G_{\Q} \to \GL_2(L)$, by a diagonal matrix, taking values in $\GL_2(\cT_{\cusp})$ and reducing to $\rho$ modulo $\gm_{\cT_{\cusp}}$. We denote this representation $\rho_{\cT_{\cusp}}$. 
 The same argument applied to $\phi^{-1}$ instead of $\phi$, yields a representation $\rho'_{\cT_{\cusp}}$ whose trace and determinant agree with those of  $\rho_{\cT_{\cusp}}$, because they can be compared in $L$, where they are equal by definition.  

In order to prove the statement about the restrictions  to $\G_{\Q_p}$, we write an exact sequence of $\cT_{\cusp}$-modules analogous to the one in \cite[(20)]{bellaiche-dimitrov} and adapt the argument as follows. The fact that    $\rho_{\mid \I_{\Q_p}}$ (resp. $\rho'_{\mid \I_{\Q_p}}$) has an infinite image and  admits a unique $\I_{\Q_p}$-stable line,   shows that the last term of that exact sequence is a monogenic $\cT_{\cusp}$-module, hence it is free, because  it is generically free of rank $1$ and $\cT_{\cusp}$ is reduced. 
\end{proof}

\subsection{Etaleness of \texorpdfstring{$\cC_{\cusp}$}{} over \texorpdfstring{$\cW$}{} at  \texorpdfstring{$f$}{}}

In this section we give a proof of Theorem~\ref{main-thm}(i). 

Let $ \varLambda$ be the completed strict local ring of $\cW$ at $\kappa(f)$. The weight map $\kappa$ induces a finite flat homomorphism  $\kappa^{\#}: \varLambda \to \cT_{\cusp}$ of local reduced complete rings. The local ring at $f$ of the fibre $\kappa^{-1}(\kappa(f))$ is a local Artinian
$\bar\Q_p$-algebra given by $\cT_{\cusp}^0=\cT_{\cusp}/\gm_{\varLambda}\cT_{\cusp}$.
We will prove in this section that $\cT_{\cusp}$ is \'etale over $\varLambda$, or equivalently that $\cT_{\cusp}^0\simeq \bar\Q_p$.

 According to Proposition~\ref{rhott} we have $(\rho_{\cT_{\cusp}},\rho_{\cT_{\cusp}}') \in \cD_{\cusp}(\cT_{\cusp})$ yielding  by functoriality a  homomorphism of local rings
 \begin{align}\label{R^cusp > T}\varphi_{\cusp}:\cR_{\cusp} \to \cT_{\cusp}. \end{align}
As  Langlands' correspondence relates the determinant to the central character,
the homomorphism \eqref{R^cusp > T} is $\varLambda$-linear (see for example \cite[Proposition~6.11]{bellaiche-dimitrov}).  Reducing modulo $\gm_{\varLambda}$ yields: 
  \[
  \varphi_{\cusp}^0:\cR_{\cusp}^0 \to \cT_{\cusp}^0. 
  \]
\begin{thm}\label{isom-cusp}
\begin{enumerate} 
\item Both  $\varphi_{\cusp}$ and $\varphi_{\cusp}^0$ are isomorphisms.
\item The  homomorphism $\kappa^{\#}:\varLambda \to \cT_{\cusp}$ is  an isomorphisms.
\end{enumerate}
\end{thm}

\begin{proof}
(i) The local ring $\cT_{\cusp}$ is topologically generated over $\varLambda$ by $U_p$ and $T_\ell$ for $\ell\nmid Np$.
A direct consequence of Proposition~\ref{rhott} is that all those elements  belong to the image of $\varphi_{\cusp}$ proving its
surjectivity. Since by Proposition~\ref{cuspidal tangent}
 the dimension of the tangent space of $\cR_{\cusp}$ is  $1$ and the local ring $\cT_{\cusp}$ is equidimensional of dimension $1$,  the surjective 
homomorphisms $\varphi_{\cusp}$ is necessarily an  isomorphism of complete local  regular rings. Reducing modulo $\gm_{\varLambda}$ proves that
$\varphi_{\cusp}^0$ is an isomorphism as well. 

(ii) By  Proposition~\ref{cuspidal tangent} the tangent space of  $\cR_{\cusp}^0$ is trivial, hence   $\cT_{\cusp}^0$ has trivial tangent space as well, 
{\it i.e.}   $\cT_{\cusp}^0\simeq \bar\Q_p$. Hence $\varLambda\to \cT_{\cusp}$ is unramified and therefore it is  an isomorphism, because both $\varLambda$ and $\cT_{\cusp}$ are complete local rings with  same residue field.
\end{proof}

\subsection{Eisenstein components  containing \texorpdfstring{$f$}{}}\label{sec:Eisenstein}

The eigencurve $\cC$ has two irreducible components corresponding to Eisenstein families containing $f$. In this subsection, we relate the completed strict local rings of these components at $f$ to  universal deformation rings.

By Proposition~\ref{def-red-rho}, one can attach to  $\cE_{\mathbf{1},\phi}$ a $\G_\Q$-reducible deformation of $\rho$:
\begin{align}\label{eq:eis-defn}
\rho_{\eis}:\G_{\Q} \to \GL_2(\varLambda), \text{ such that  } \tr(\rho_{\eis})=\phi\chi_{p}+\mathbf{1} \text{ and } \det(\rho_{\eis})=\phi\chi_{p}.
\end{align}
The irreducible component of $\cC$ corresponding to the Eisenstein family $\cE_{\mathbf{1},\phi}$ is \'etale over the weight space, hence  the completed strict local ring $\cT_{\rho}^{\eis}$ of this  component at $f$  is isomorphic to $\varLambda$.

Since $\rho_{\eis}$ admits a trivial (hence unramified at $p$) rank $1$ quotient, it defines an element of
$\cD^{\eis}_\rho(\varLambda)$ (see  Definition~\ref{defn:eis}), {\it i.e.} a  $\bar\Q_p$-algebra homomorphism
\begin{align}\label{Resi=Teis} \cR^{\eis}_\rho \to \varLambda.
 \end{align}

\begin{lemma}\label{Reis=A}
The map \eqref{Resi=Teis} is an isomorphism of $\varLambda$-algebras.
\end{lemma}

\begin{proof} We explained in \S\ref{ord-def} that the $\varLambda$-algebra structure of $\cR^{\ord}_\rho$
comes from the determinant, and the same is true for its quotient $\cR^{\eis}_\rho$. As $\det \rho_{\eis}=\phi \chi_{p}$
it follows that  \eqref{Resi=Teis} is a $\varLambda$-algebra homomorphism, in particular it is  surjective.  
By  Lemma~\ref{lemmatd} the tangent space of $\cR^{\eis}_\rho$ is $1$-dimensional,  hence   \eqref{Resi=Teis} is  an isomorphism.
\end{proof}

Let $\cT^{\eis}_{\rho'}$ be the completed strict local ring of the irreducible component of $\cC$ corresponding to the Eisenstein family $\cE_{\phi,\mathbf{1}}$. Similarly to Definition~\ref{defn:eis} let $\cD^{\eis}_{\rho'}$ be the sub functor of $\cD_{\rho'}^{\ord}$ 
consisting of deformations which are reducible and ordinary for the same filtration.
An argument  as in Lemmas~\ref{idealJ=C} and \ref{Reis=A} with $\rho' \otimes \phi^{-1}$ (resp. $\phi^{-1}$) replaced by   $\rho$ (resp. $\phi$) shows  that $\cD^{\eis}_{\rho'} $ is pro-representable by a regular ring 
$\cR^{\eis}_{\rho'}\xrightarrow{\sim} \varLambda \xrightarrow{\sim} \cT^{\eis}_{\rho'}$, and that the structural homomorphism
 $\kappa^{\#}: \varLambda\xrightarrow{\sim} \cT^{\eis}_{\rho'}$ is an   isomorphism.

\section{The theorem of Ferrero--Greenberg and modularity}\label{Iwasawa Cohomology and Eisenstein Ideal}

 We use the notations from \S\ref{iwasawa-coh}. 
For any odd  Dirichlet character $\psi \ne \omega^{-1}_p$, the Kubota--Leopoldt $p$-adic $L$-function  $L_p(\psi \omega_p,s)$ is  analytic in  $s \in \Z_p$  and characterised by the interpolation property that for all integers $k\geqslant 1$ such that  $\omega_p^{k-1}=\mathbf{1}$  one has: 
\begin{align}
\label{interpolation Lp}
L_p(\psi\omega_p, 1-k)=(1-\psi(p) p^{k-1})L(\psi,1-k). 
\end{align}
As $\phi(p)=1$, it follows that $L_p(\phi\omega_p,s)$ has a trivial zero at $s=0$. 
As well known (see \cite[Theorem 7.10]{washington-GTM}), there exists an element  $\zeta_{\phi}(X)  \in \Z_p \lsem X \rsem \simeq \Z_p \lsem 1+p^\nu \Z_p\rsem$ such that for $k\in\Z_{\geqslant 1}$:
\begin{align}
\label{Kubota--leopoldt}
\zeta_{\phi}((1+p^\nu)^{k-1}-1)=L_p( \phi \omega_p,1-k).
\end{align}

\subsection{The Eisenstein ideal}\label{subsection: The Eisenstein ideal}
In this section we study the congruences between the Eisenstein family $\cE_{\mathbf{1},\phi}$ and the unique cuspidal family $\cF$ specialising to  $f$. To achieve our goal we define an appropriate quotient of the completed strict local ring $\cT$  of $\cC$ at $f$  which has two generic points, one corresponding to $\cF$  and the other to   $\cE_{\mathbf{1},\phi}$. We let  $\cT^{\ord}_{\rho}$ be the image of the abstract Hecke algebra $\varLambda[U_p, T_\ell]_{\ell\nmid pN}$ in 
 $\cT_{\cusp}\times_{\bar\Q_p} \varLambda^{\eis}$, where  $\varLambda^{\eis}=\varLambda$ is the Eisenstein Hecke algebra corresponding to $\cE_{\mathbf{1},\phi}$. By definition  $\cT^{\ord}_{\rho}$ is a reduced local quotient  of $\cT$,  and  is in fact a  $\varLambda$-sub-algebra of $\cT_{\cusp} \times_{\bar\Q_p} \varLambda^{\eis}$  surjecting to each one of the  factors.  

Since $\varLambda$ is a discrete valuation ring,  Hida's congruence module yoga gives an integer $e\geqslant 1$, such that  $\cT^{\ord}_{\rho}= \cT_{\cusp} \times_{\bar\Q_p[X]/(X^e)} \varLambda^{\eis}$. 
Indeed, if one  considers  the natural projections 
\[\pi_{\cusp}: \cT^{\ord}_{\rho}\twoheadrightarrow \cT_{\cusp}\text{  and }
\pi_{\eis}: \cT^{\ord}_{\rho}\twoheadrightarrow \varLambda^{\eis}\] 
and let $J_{\eis}= \pi_{\cusp}(\ker(\pi_{\eis}))$ and $(X^e)=\pi_{\eis}(\ker(\pi_{\cusp}))$, one obtains a Cartesian diagram
\begin{align} \label{lem:cong-module} 
\xymatrix@C=3em@R=2em{  & \ker(\pi_{\cusp})\ar@{_{(}->}[d]\ar^{\sim}[r]& (X^e) \ar@{_{(}->}[d]\\ 
\ker(\pi_{\eis}) \ar@{^{(}->}[r] \ar_{\sim}[d]& \cT^{\ord}_{\rho} \ar@{->>}^{\pi_{\eis}}[r]\ar@{->>}_{\pi_{\cusp}}[d] & \varLambda^{\eis} \ar@{->>}[d]\\
J_{\eis} \ar@{^{(}->}[r]  & \cT_{\cusp}\ar@{->>}[r] &  \cT_{\cusp}/J_{\eis}\xrightarrow{\sim} \varLambda^{\eis}/(X^e), }
\end{align} 
justifying the $\varLambda$-algebras  isomorphisms: 
\[ \cT_{\cusp}/J_{\eis}\xrightarrow{\sim} \varLambda^{\eis}/(X^e) \text{  and }  
\cT^{\ord}_{\rho}=\left\{(a,b)\in  \cT_{\cusp} \times\varLambda^{\eis}\mid (a\!\mod{J_{\eis}})= (b\!\mod{X^e}) \right\}. \]

\begin{lemma}
One  has  $\mathrm{Ann}(\ker(\pi_{\eis}))=\ker(\pi_{\cusp})$. 
\end{lemma}
\begin{proof} The above description of $\cT_{\cusp}$ as fibre product implies that $\ker(\pi_{\cusp})  \subset \mathrm{Ann}(\ker(\pi_{\eis}))$. 
Suppose that $(a,b)\in \mathrm{Ann}(\ker(\pi_{\eis}))\subset \cT_{\cusp} \times \varLambda^{\eis}$. If $a\neq 0$, then there would exist a 
cuspidal  family $\cG$ such that $a\cdot \cG\neq 0$. The Eisenstein ideal $\ker(\pi_{\eis})$ would then annihilate the 
cuspidal  family $a \cdot \cG$, meaning that $\cT^{\ord}_{\rho}$ would act on both $a \cdot \cG$ and on $\cE_{\mathbf{1},\phi}$ through its
quotient $\varLambda^{\eis}$.  Hence $a \cdot \cG$ is a  $\varLambda$-adic eigenform having the same eigenvalues as $\cE_{\mathbf{1},\phi}$. This is impossible as cuspidal eigenforms have irreducible Galois representations. 
\end{proof}

\begin{prop} \label{e=1}
One has $e=1$ and $\cT^{\ord}_{\rho}\xrightarrow{\sim}  \varLambda \times_{\bar\Q_p} \varLambda$.
\end{prop}
\begin{proof} By Theorem~\ref{isom-cusp} one knows that $\cT_{\cusp}\xrightarrow{\sim} \varLambda$. Moreover the image of 
$U_p$ in  $\varLambda^{\eis}$ is $1$, whereas by the proof of Proposition~\ref{cuspidal tangent} its image in $\cT_{\cusp}/(X^2)\simeq \cR_{\cusp}/(X^2)$ equals $(1-\cL(\phi) X)$ with 
$\cL(\phi) \neq 0$, where $X$ is a topological generator of $\cT_{\cusp}$ lifting the element of $t_{\cusp}$ corresponding to $\mu=1$. 
\end{proof} 

\subsection{The full eigencurve and a duality}\label{full eigen}

Let $\cC^{\full}$ be the $p$-adic eigencurve over $\cW$ of tame level $N$ constructed 
using the Hecke operators $T_\ell$ for  $\ell \nmid Np$ and $U_\ell$ for $\ell \mid Np$. 
It is endowed with a locally finite surjective morphism  $\cC^{\full} \to \cC$.
Using the relations between  abstract Hecke operators and  the fact that  the diamond operators
at  all  $\ell \nmid Np$ belong to $\cO_{\cW}(\cW)$, one sees that $T_n\in \cO_{\cC^{\full}}(\cC^{\full})$ for
all $n\geqslant 1$. 
 There is a natural bijection between $\cC^{\full}(\bar\Q_p)$ and the set of systems of eigenvalues of overconvergent eigenforms with finite slope, tame level dividing $N$ and weight in $\cW(\bar\Q_p)$,  sending $g$ to $\{(T_\ell(g))_{\ell \nmid Np}, (U_\ell(g))_{\ell \mid Np}\}$. Let $\cT^{\full}$ be the completed strict local ring of $\cC^{\full}$ at $f$.
Let $\cC_{\cusp}^{\full}$ be the closed analytic subspace of $\cC^{\full}$ corresponding to cuspidal overconvergent modular forms and let $\cT^{\full}_{\cusp}$ be the completed strict local ring of $\cC_{\cusp}^{\full}$ at  $f$.

\begin{rem}\label{T_ord  v2}
By construction, the ordinary locus of $\cC^{\full}$ (resp. $\cC_{\cusp}^{\full}$) is isomorphic to the rigid analytic space attached to the maximal spectrum of the generic fibre of the $p$-ordinary Hecke algebra (resp. $p$-ordinary cuspidal Hecke algebra) of tame level $N$ constructed by Hida \cite{hida85}. 
As  Galois  orbits of cuspidal Hida families  are in bijection with  irreducible components of the ordinary locus of  $\cC_{\cusp}^{\full}$,   Theorem~\ref{isom-cusp} combined with  Proposition~\ref{monodromy} below shows that there exists a unique, up to Galois conjugacy, cuspidal Hida family  specialising to $f$ (see \cite{dim-durham} for more details). 
\end{rem}

\begin{prop} \label{monodromy} 
The  natural injection $\cT_{\cusp} \hookrightarrow \cT_{\cusp}^{\full}$ is an isomorphism.
Moreover, for any  prime  $\ell$ dividing $N$,  $\rho_{\cT_{\cusp}}(\I_{\Q_\ell})$ is finite,  the $\cT_{\cusp}$-module $(\rho_{\cT_{\cusp}})^{\I_{\Q_\ell}}$ is free of rank $1$ and  $\Frob_\ell$ acts on it as $U_\ell$.  
\end{prop}
\begin{proof} 
 By  \cite[Lemma~4.3.7]{bellaiche-chenevier-book} there exists an open  admissible affinoid neighbourhood $\cV$ of $f$ in $\cC_{\cusp}$ endowed with a representation $\rho_\cV:\G_\Q \to \GL_2(\cO(\cV))$ extending $\rho_{\cT_{\cusp}}$. 
Assuming  $\cV$ to be connected which one always can,  \cite[Lemma~7.8.17]{bellaiche-chenevier-book} guarantees that the restriction of the pseudo-character $\tr(\rho_\cV)$ to $\I_{\Q_\ell}$ is constant and equal to $\phi+\mathbf{1}$. 
 As  $\phi_{\mid \I_{\Q_\ell}} \ne \mathbf{1}$,   using the description of the Weil-Deligne representation attached to the Steinberg representation,   for any classical point $g$ in $\cV$  
the monodromy operator of Weil-Deligne representation of $\rho_{g \mid \G_{\Q_\ell}}$ vanishes. 
 It follows that the monodromy operator of the Weil-Deligne representation attached to $\rho_{\cV\mid \G_{\Q_\ell}}$ in \cite[Lemma~7.8.14]{bellaiche-chenevier-book} vanishes as well, and   the restriction of  $\rho_{\cV}$ to $\I_{\Q_\ell}$ is constant equal to $\phi\oplus \mathbf{1}$.  In particular $\rho_{\cT_{\cusp}}(\I_{\Q_\ell})$ is finite and   the $\I_{\Q_\ell}$-invariants of $\rho_{\cT_{\cusp}}$ are a rank $1$ direct summand.
It remains to show that $\Frob_\ell$  acts  $(\rho_\cV)^{\I_{\Q_\ell}}$ via $U_\ell$, and thus  $U_\ell\in \cT_{\cusp}$.
By local-global compatibility at $\ell$,  any classical point $g$ in $\cV$ is new at $\ell$ and   $\rho_g(\Frob_\ell)$ acts on the line $(\rho_g)^{\I_{\Q_\ell}}$ by  the $U_\ell$-eigenvalue of  $g$   (see also \cite[Proposition~7.1]{bellaiche-dimitrov}). 
One then concludes using  Zariski density.
\end{proof}

Let  $S^{\dag}_{\gm_f}$ (resp. $M^{\dag}_{\gm_f}$) be the $\cT_{\cusp}$-module (resp. $\cT$-module) obtained by localising and completing
 $S^{\ord}_\mathcal{U}$ at the maximal ideal $\gm_f$ of  the abstract Hecke algebra $\varLambda[U_p, T_\ell]_{ \ell\nmid pN}$ corresponding to the system of Hecke eigenvalues of $f$.  In view of Proposition~\ref{monodromy}, the 
works of  Hida  \cite[\S2]{hida86} and Coleman \cite[Proposition~B.5.6]{coleman} yield a natural   isomorphism of $\varLambda$-modules:
  \begin{align} \label{perfect-duality}
    S^{\dag}_{\gm_f}\xrightarrow{\sim} \Hom_\varLambda(\cT_{\cusp},\varLambda) ,  \quad \cG\mapsto (T \mapsto a_1(T\cdot \cG)).
 \end{align}

\begin{cor}\label{unique-cusp-family} 
The generalised cuspidal eigenspace  $S^{\dag}\lsem f \rsem$
equals $\bar\Q_p f$, and $S^{\dag}_{\gm_f} =\varLambda \cdot \cF$ for a  unique normalised cuspidal eigenfamily  $\cF$. 
\end{cor}

\begin{proof} As $\cT_{\cusp}$ is canonically isomorphic to $\varLambda$ by Theorem~\ref{isom-cusp}, the identity map in the right hand side of \eqref{perfect-duality} corresponds to a normalised cuspidal eigenfamily  $\cF$, which is a basis of $S^{\dag}_{\gm_f}$ over 
$\varLambda$.  By definition $S^{\dag}\lsem f \rsem$ is isomorphic to  the $\cT/\gm_\varLambda\cT$-module $S^{\dag}_{\gm_f}/ \gm_\varLambda S^{\dag}_{\gm_f}$, hence is $1$-dimensional  and spanned by  $f$. 
\end{proof}

\begin{rem}
Note  that  $S^{\dag}\lsem f \rsem$ is spanned by  $f$ which is classical and cuspidal-overconvergent, even though it is not cuspidal as a classical form. 
This result was first conjectured in \cite[Hypothesis (C')]{DLR1} for the following arithmetic application. 
Let $E$ be an elliptic curve over $\Q$  and $g$ be a classical weight $1$ form.
 When the analytic rank of $E$ over the splitting field of $\rho_f\otimes \rho_g$ is $2$, the elliptic Stark conjecture of \cite{DLR1} relates the values of some $p$-adic iterated integrals to formal group logarithms of global points of $E$. 
 The classicality of $S^{\dag}\lsem f \rsem$ plays a crucial role in the formulation of the conjecture, as 
 the definition of the $p$-adic iterated integrals relies on that assumption. The scenario in which $f$ and $g$ are both  irregular Eisenstein series  is particularly appealing because the  splitting field is then  cyclotomic.  Numerical evidence towards this conjecture are given in  \S7 of {\it loc. cit.} and  recently   Rotger--Casazza  \cite{casazza-rotger} established a result in that direction, under the 
 hypotheses   (C) and  (C')  from {\it loc. cit.} which can now be waived thanks to  Corollary~\ref{unique-cusp-family}. 
\end{rem}

\subsection{On the constant term of  Eisenstein families}
The aim of this section is to show that the Kubota--Leopoldt $p$-adic $L$-function $L_p(\phi\omega_p,s)$  has a simple trivial zero 
at $s=0$,  or equivalently  that  $\zeta_\phi(X) \in \varLambda$ vanishes at order $1$ at $X=0$. From the interpolation property 
\eqref{interpolation Lp} we know that $\zeta_\phi(0)=0$,  so it suffices to prove that
the order of vanishing is at most $1$. To achieve this we first relate $\zeta_\phi(X)$ to the constant term of the Eisenstein family  $\cE_{\mathbf{1},\phi}$.

 Let $\cU$ be an open admissible affinoid of $\cW$ containing  the weight  $\kappa(f)$ of $f$. 
 By localising and completing \eqref{eval-at-cups2} at the maximal ideal corresponding to $\kappa(f)\in\cU$, we 
obtain an exact sequence  of flat $\varLambda$-modules:
\begin{align}\label{residueniveau local}
0 \to S^{\ord}_{\varLambda} \to M^{\ord}_{\varLambda} \xrightarrow{\res_\varLambda} C_\varLambda \to 0,
\end{align}
where $C_\varLambda$ is a  direct factor  of $\displaystyle \prod_{[\delta]\in \Gamma_1(N)\backslash \mathbb{P}^1(\Q)} \varLambda$.  

Let $A_\delta(\mathbf{1}, \phi)$, resp. $A_{\delta}(\phi, \mathbf{1})$, be the constant terms of $\cE_{\mathbf{1},\phi}$, resp. $\cE_{\phi,\mathbf{1}}$,  at  $[\delta]\in \Gamma_1(N)\backslash \mathbb{P}^1(\Q)$.  
\begin{prop}\label{q-expa family}\
\begin{enumerate}
\item One has $A_\infty(\mathbf{1}, \phi)=\tfrac{1}{2}\zeta_\phi$,  $A_0(\mathbf{1}, \phi)=0$ and 
$A_\delta(\mathbf{1}, \phi)\in \varLambda \cdot \zeta_\phi$  for all $\delta$. 
\item One has $A_\infty(\phi, \mathbf{1})=0$, $A_0(\phi, \mathbf{1})\in \varLambda^\times \zeta_{\phi^{-1}}$ and 
 $A_\delta(\phi,\mathbf{1})\in \varLambda \cdot \zeta_{\phi^{-1}}$ for all $\delta$. 
\end{enumerate}
\end{prop}

\begin{proof}
(i)  We will establish the Proposition via a computation of the constant term of the  specialisations of the Eisenstein families  $\cE_{\mathbf{1},\phi}$ and $\cE_{\phi,\mathbf{1}}$ at all classical weights $k\geqslant 2$ such that $\omega_p^{k-1}=\mathbf{1}$. 
Using the notations from \S\ref{basic-eigencurve}, the set $\Gamma_1(N) \backslash \mathbb{P}^1(\Q)$ is in bijection with the $\Gamma_1(N)$-orbits of cusps of $X$. Similarly, the sets $(\Gamma_1(N)\cap\Gamma_0(p)) \backslash \mathbb{P}^1(\Q)$ and 
$(\Gamma_1(N)\cap\Gamma_0(p))  \backslash  (\Gamma_0(p) \infty)$ are in bijection with
the $\Gamma_1(N)\cap\Gamma_0(p)$-orbits of cusps in  $X_{\mathrm{Iw}}$ and in the multiplicative locus 
$X^\times_{\mathrm{Iw}}(0)$, respectively.
On the level of cusps the canonical section \eqref {can-section} is simply 
given by the inverse of the natural bijection 
\begin{align}\label{bijection-cusps}
(\Gamma\cap\Gamma_0(p))  \backslash  (\Gamma_0(p) \infty)\xrightarrow{\sim} \Gamma \backslash \mathbb{P}^1(\Q)
\end{align}
allowing us to represent any  $\Gamma_1(N)$-orbit of cusps of $X$ with an element $\delta=\left [\begin{smallmatrix} a\\c \end{smallmatrix} \right ]  \in  \Gamma_0(p) \infty$ such that  the integers $a$ and  $c$ are relatively prime  and $p\mid c$. 

For every $g \in  M_k(\Gamma_1(N),\phi)$, we define $g^{(p)}(z)=(g-g_{| \iota})(z)= g(z)- p^{k-1} g(pz)\in M_k(\Gamma,\phi)$, where  $\iota= \left ( \begin{smallmatrix} p & 0\\ 0 & 1 \end{smallmatrix} \right ) $. 
The modular form $g$ (resp. $g^{(p)}$) can be evaluated at cusps of $X$ (resp. $X_{\mathrm{Iw}}$); moreover by 
$\Gamma_1(N)$-invariance (resp. $(\Gamma_1(N)\cap\Gamma_0(p))$-invariance) of $g$ (resp. $g^{(p)}$) the value is well-defined on 
$\Gamma_1(N)$-orbits (resp. $(\Gamma_1(N)\cap\Gamma_0(p))$-orbits) of cusps. 
We first compute the $q$-expansion of $g^{(p)}$ at the  cusp  $\delta$, in terms of $q$-expansions of $g$. 
 Choose a matrix  $ \gamma_\delta=\left(\begin{smallmatrix} a & b\\ c & d \end{smallmatrix}\right ) \in \rm{SL}_2(\Z)$. 
The $q$-expansion of $g^{(p)}$ at the cusp $\delta$ is then given by 
\[g^{(p)}_{|\gamma_\delta}(z)=g_{|\gamma_\delta}(z)-g_{|\iota\gamma_\delta}(z)=g_{|\gamma_\delta}(z)-g_{|\gamma_{p\delta}\iota}(z)=g_{|\gamma_\delta}(z)-p^{k-1}g_{|\gamma_{p\delta}}(pz),\]
where $\gamma_{p\delta}=\left (
\begin{smallmatrix} a & bp\\ c p^{-1}& d \end{smallmatrix} \right ) \in \rm{SL}_2(\Z)$. 
Letting $A_\delta$ (resp. $A_{\delta}^{(p)}$) denote the constant term  of $g$ (resp. $g^{(p)}$) at the  cusp $\delta \in \mathbb{P}^1(\Q)$ (resp. $\delta\in \Gamma_0(p) \infty$) we have  \[A_{\delta}^{(p)}=A_\delta-p^{k-1}A_{p\delta}.\] 

Let now $g$ be the weight $k$ Eisenstein series $E_k( \mathbf{1},\phi) \in M_k(\Gamma_1(N),\phi)$. Then $g^{(p)}$ in the above notation is the ordinary $p$-stabilisation of $g$, and it is also the weight $k$ specialisation of the Eisenstein family $\cE_{\mathbf{1},\phi}$. 
By \cite[Proposition~1.1]{ozawa}, for 
$\delta=\left [\begin{smallmatrix} a\\c \end{smallmatrix}\right ] $
we have 
\begin{align*}
A_\delta=0 \text{ , if } N \nmid c, \text{ and }\qquad  A_\delta=\frac{\phi^{-1}(|a|)}{2}L(\phi, 1-k) \text{ , if } N \mid c. 
\end{align*}
Thus, $A_{\delta}^{(p)}=0$ if $N \nmid c$, whereas if $N \mid c$ then using \eqref{interpolation Lp} and 
$\phi (p)=1$ one finds
\[
A_{\delta}^{(p)}=A_\delta-p^{k-1}A_{p\delta}=(1-p^{k-1})\frac{\phi^{-1}(|a|)}{2}L(\phi, 1-k)=\frac{\phi^{-1}(|a|)}{2} L_p(\phi \omega_p, 1-k). \] 
As the weights $k\geqslant 2$ such that $\omega_p^{k-1}=\mathbf{1}$ are  Zariski dense in the connected component of $\cW$ containing $\kappa(f)$,  we deduce that $A_\delta(\mathbf{1}, \phi)$ equals $\tfrac{\phi^{-1}(|a|)}{2} \zeta_\phi(X)$ if $N \mid c$, and  vanishes 
otherwise. In particular, $A_\infty(\mathbf{1}, \phi)=\tfrac{1}{2}\zeta_\phi$ and  $A_0(\mathbf{1}, \phi)=0$. 

 (ii) Let $g=E_k(\phi, \mathbf{1})\in M_k(\Gamma_1(N),\phi)$ and  let $\delta$ be as in (i). By \cite[Proposition~1.1]{ozawa} 
\begin{align*}
A_\delta=0,  \text{ if } (N,c)\ne 1, \text{ and }\qquad  A_\delta=-\frac{\tau(\phi)}{2N^k}\phi(|c|)L(\phi^{-1}, 1-k),  \text{ if } (N,c)=1, 
\end{align*}
 where $\tau(\phi)$ denotes the Gauss sum of $\phi$. 
Thus, if $(c, N)\ne 1$ then $A_{\delta}^{(p)}$ vanishes, whereas if  $(c, N)=1$ then  one finds
\begin{align*}
A_{\delta}^{(p)}=A_\delta-p^{k-1}A_{p\delta}&=-\frac{\tau(\phi)}{2N^k}\phi(|c|)(1-p^{k-1})L(\phi^{-1}, 1-k)=-\frac{\tau(\phi)}{2N^k}\phi(|c|) L_p(\phi^{-1} \omega_p, 1-k).
\end{align*}
The form $g^{(p)}$ is the weight $k$ ordinary specialisation of $\cE_{\phi,\mathbf{1}}$. 
Since $(p, N)=1$ and $\omega_p^{k-1} =\mathbf{1}$, $N^k$ is the weight $k$ specialisation of an element  in $\varLambda^\times$, while $L_p(\phi^{-1} \omega_p, 1-k)$ is the weight $k$ specialisation of $\zeta_{\phi^{-1}}(X)$. The claim then easily follows. 
\end{proof}

We state  a generalisation of a  result of Wiles \cite[Theorem~4.1]{wiles-imc} and Ohta \cite[Corollary~A.2.4]{ohta-eis} to the case of a trivial zero, and give an alternative proof of a famous result of Ferrero and Greenberg \cite{ferrero-greenberg}.

\begin{prop} \label{ordL_p} There exists an isomorphism of $\varLambda$-algebras
$\cT_{\cusp}/J_{\eis} \xrightarrow{\sim} \varLambda/(\zeta_\phi(X))$. 

The Kubota--Leopoldt  $p$-adic $L$-function $\zeta_\phi(X)$ has a simple  zero  at $X=0$.
\end{prop}

\begin{proof} 
It follows from Proposition~\ref{q-expa family} and the exact sequence \eqref{residueniveau local} that  $\cE_{\mathbf{1},\phi} \mod (\zeta_\phi(X))$ is a cuspidal family. Using \eqref{perfect-duality} and the freeness of $\cT_{\cusp}$ over $\varLambda$ one has 
\[S^{\dag}_{\gm_f}\otimes_\varLambda (\varLambda/(\zeta_\phi(X)))= \Hom_\varLambda(\cT_{\cusp}, \varLambda/(\zeta_\phi(X))). \]
 As $\cE_{\mathbf{1},\phi}$  is an eigenfamily, the  map $\cT_{\cusp}\to \varLambda/(\zeta_\phi(X))$ resulting from the above equation yields (via Lemma~\ref{Reis=A}) a $\varLambda$-algebra  homomorphism $\cT_{\cusp}\to \varLambda^{\eis}/(\zeta_\phi(X))$ . 
As $e$ defined above \eqref{lem:cong-module} is the largest integer such that the projection $ \cT^{\ord}_\rho\to \varLambda^{\eis}/(X^e)$ factors through $\cT_{\cusp}$, it follows that $\zeta_\phi(X)$ divides $X^e$. 
The claims in the Proposition then follow from the fact that $e=1$ by  Proposition~\ref{e=1}. 
\end{proof}

\subsection{An application of Wiles' numerical criterion} \label{s:ord-modularity}
 The goal of this subsection is to prove that there exists an isomorphism   $\varphi: \cR^{\ord}_{\rho} \xrightarrow{\sim} \cT^{\ord}_\rho$ of complete intersections. 
 Lemma~\ref{idealJ=C} and Proposition~\ref{rhott}   
  yield the following   homomorphism of $\varLambda$-algebras:
\begin{align}\label{RarrowT} 
\varphi: \cR^{\ord}_{\rho} \to \cR_{\cusp} \times_{\bar\Q_p} \cR^{\eis}_\rho \to 
 \cT_{\cusp} \times_{\bar\Q_p} \varLambda=\cT^{\ord}_\rho.
\end{align}
which is surjective by the same argument as in the proof of Theorem~\ref{isom-cusp}. To prove its injectivity
we appeal to  a variant of Wiles' numerical criterion due to Lenstra \cite{lenstra}. 
\begin{thm}\label{numerical criterion}
Let  $\varphi: R \to  T$ be a surjective homomorphism of  local $\varLambda$-algebras. 
Suppose that $T$  is  finite and flat as $\varLambda$-module and let  $\pi: T \to \varLambda$ be a $\varLambda$-algebra homomorphism. 
 Let  $J=\ker(\pi\circ\varphi)$  and assume that $\eta_T=\pi(\mathrm{Ann}(\ker\pi)) \neq 0$.
Then \[
\mathrm{length}_\varLambda(J/J^2 ) \geqslant \mathrm{length}_\varLambda(\varLambda/\eta_T)\]
and the  equality holds if and only if $\varphi$ is an isomorphism and $T$ is a  complete intersection.
\end{thm}

We wish to apply the above criterion to  \eqref{RarrowT}  and  $\pi_{\eis}: \cT^{\ord}_\rho \to \varLambda$.  
  By Proposition~\ref{e=1}   
\[\label{eta_T}
\eta_T= \pi_{\eis}(\mathrm{Ann}(\ker(\pi_{\eis})))=\pi_{\eis}(\ker(\pi_{\cusp}))=(X^e)=(X). 
\]
hence the $\varLambda$-module $\varLambda/\eta_T$ has length $1$.
  In order to apply the numerical criterion, we must  compute the length of $J/J^2$ over $\varLambda$, where $J=\ker(\pi_{\eis}\circ\varphi)$.

\begin{prop} \label{J} The $\varLambda$-module $J/J^2$ is torsion of length at most $1$.
\end{prop}
\begin{proof}
Note first that because $\cR^{\ord}_\rho$ is Noetherian,  $J/J^2$ is a module of finite type over $\cR^{\ord}_\rho/J\simeq\varLambda$. The claim is  equivalent to showing that for all $n\geqslant 1$ one has:
\begin{align}\label{eq:length}
 \mathrm{length}_{\varLambda}(\Hom(J/J^2, \varLambda/(X^n)))\leqslant1. 
 \end{align}
To show this we will employ a   Galois cohomology argument. Let us 
   fix  a lift $ \rho_{\cR}=\left( \begin{smallmatrix}  A&B\\ C & D \end{smallmatrix} \right):\G_{\Q}\to \GL_2(\cR^{\ord}_{\rho})$ representing the universal ordinary deformation in an ordinary basis. 
By Lemmas~\ref{idealJ=C} and \ref{Reis=A} the  ideal $J$ of $\cR^{\ord}_{\rho}$ is generated by the set of $C(g)$ for $g \in \G_{\Q}$. Since
\[ \rho_{\cR} \otimes \cR^{\ord}_{\rho}/J \simeq {\rho}_{\varLambda}
\simeq \left( \begin{smallmatrix} 
\Phi & \ast\\ 
0 & \mathbf{1} \end{smallmatrix} \right)
\]
 in the ordinary basis, a direct computation shows that  the function 
\begin{align}
\bar{C}: \G_\Q \to J/J^2 ,\,\, g \to \Phi^{-1}(g) C(g)\mod{J^2}
\end{align}
belongs to $\ker\left(Z^1(\G_{\Q}^{Np}, \Phi^{-1}\otimes_\varLambda (J/J^2))\to
Z^1(\Q_p, \Phi^{-1}\otimes_\varLambda (J/J^2))\right)$. As the image of $\bar{C}$ contains 
a set of generators of $J/J^2$ as a module over $\cR^{\ord}_{\rho}/J \simeq \varLambda$, 
the natural map $h \mapsto h\circ \bar{C}$ 
\begin{align*}
\Hom_\varLambda(J/J^2, \varLambda/(X^n))\hookrightarrow 
\ker \left (Z^1(\G_{\Q}^{Np},\Phi_{n}^{-1})  \to Z^1(\Q_p,\Phi_{n}^{-1})\right )
\end{align*}
is injective for all $n\geqslant 1$. Moreover by Proposition~\ref{def-red-rho}(ii) one has: 
\begin{align*}
\ker \left (Z^1(\G_{\Q}^{Np},\Phi_{n}^{-1})  \to Z^1(\Q_p,\Phi_{n}^{-1})\right )
=\ker \left (B^1(\G_{\Q}^{Np},\Phi_{n}^{-1} ) \to B^1(\Q_p,\Phi_{n}^{-1})\right )
\end{align*}
Since $\Phi_{1|\G_{\Q_p}}=\phi_{|\G_{\Q_p}}=\mathbf{1}$, while 
$\Phi_{2|\G_{\Q_p}}\neq\mathbf{1}$, the latter is given by the length $1$
$\varLambda$-module  $(X^{n-1})/(X^n)\simeq \bar\Q_p$. 
This proves \eqref{eq:length}, hence the Proposition.
 \end{proof}

\begin{thm}\label{t:ord-modularity}
$\varphi: \cR^{\ord}_{\rho} \twoheadrightarrow \cT^{\ord}_\rho$ is an isomorphism of  complete local intersections over $\varLambda$.
\end{thm}

\begin{proof} The claim is a direct consequence of  Theorem~\ref{numerical criterion} in view of Proposition~\ref{J}.\end{proof}

\section{Local structure of the eigencurve at \texorpdfstring{$f$}{}}
In this section we will first complete the proof of Theorem~\ref{main-thm}, then we will prove Theorem~\ref{qexpansionoverc} using the methods of \cite{DLR3} and \cite{DLR4}.
We recall that we have two Eisenstein families $\cE_{\mathbf{1},\phi}$ and $\cE_{\phi,\mathbf{1}}$, as well as a 
cuspidal family $\cF$ (see Corollary~\ref{unique-cusp-family})  containing $f$. We also recall that  $\varLambda=\bar\Q_p\lsem X \rsem $ denotes the universal deformation ring of $\phi=\det(\rho)$.

\subsection{Failure of  Gorensteinness  of \texorpdfstring{$\cC$}{} at \texorpdfstring{$f$}{}}
 By Theorem~\ref{isom-cusp} and Lemma
\ref{Reis=A} the  structural homomorphisms 
$\varLambda\xrightarrow{\sim} \cT_{\cusp}$, $\varLambda\xrightarrow{\sim} \cT^{\eis}_{\rho}$ and $ \varLambda\xrightarrow{\sim} \cT^{\eis}_{\rho'}$ are  isomorphisms. 
 Since the eigencurve $\cC$ is reduced, it results from the above discussion a canonical  inclusion of local $\varLambda$-algebras:
\begin{align}\label{injection}
\pi= (\pi_{\cusp},\pi^{\eis}_{\rho},\pi^{\eis}_{\rho'}):\cT \hookrightarrow \varLambda \times_{\bar\Q_p} \varLambda\times_{\bar\Q_p} \varLambda, 
\end{align}
where $\cT$ denotes the completed strict local ring of $\cC$  at $f$. 

Consider the sub-algebra
$\cT'$ of $\cT$ generated over $\varLambda$ by $T_\ell$, $\ell\nmid Np$. 

\begin{thm}\label{prod-fibr} The image of $\cT'$  under the natural inclusion \eqref{injection} is given by 
\begin{align}\label{eq:trace-image}
\left\{(a,b,c)\in \varLambda \times_{\bar\Q_p} \varLambda\times_{\bar\Q_p} \varLambda \Big{|} 
\left(\cL(\phi^{-1})+\cL(\phi)\right)a'(0)=\cL(\phi^{-1})b'(0)+\cL(\phi)c'(0)  \right\}.
\end{align}

Moreover, $\pi$ defined in  \eqref{injection} is   an isomorphism, $X \cT$ is an ideal of $\cT'$,  and 
 \[\cT/X\cT =(\cT'/X\cT)[U_p]/(U_p-1)^2.\] 
\end{thm}

\begin{proof} We first show that $\pi(\cT')\supset (X^2\cdot \varLambda)^3$, by showing that
$\pi(\cT')$ surjects on each of the $3$ double products $\varLambda\times_{\bar\Q_p} \varLambda$ in  \eqref{injection}, 
and by observing that
$(0,b,X)\cdot(a,0,X)=(0,0,X^2)$, etc. 
The surjectivity of the homomorphism $\cR^{\ps}\to \cR^{\ord}_{\rho}$ established in the proof of 
Lemma~\ref{lem-surj}, implies the surjectivity of the composed map $\cT'\to \cT\to \cT^{\ord}_{\rho}$, {\it i.e.}
$(\pi_{\cusp},\pi^{\eis}_{\rho})$ is surjective, and similarly $(\pi_{\cusp},\pi^{\eis}_{\rho'})$ is 
surjective as well. As for the surjectivity of $(\pi^{\eis}_{\rho},\pi^{\eis}_{\rho'})$ it suffices to show that
the image of 
\[\tr(\rho_{\eis})-\tr (\rho'_{\eis})=(\phi-\mathbf{1})\cdot(\chi_{p}-\mathbf{1}):\G_\Q\to \varLambda,\]
contains an element of valuation $1$, where $\rho_{\eis}$ is the  Galois deformation of $\rho$ attached to 
 $\cE_{\mathbf{1},\phi}$ (see  \eqref{eq:eis-defn}), whereas $\rho'_{\eis}$ is the  Galois deformation of $\rho'$ attached to $\cE_{\phi,\mathbf{1}}$. This follows easily from the fact that
 the abelian extensions $H$ and $\Q_\infty$ of $\Q$ (which are the fixed fields of   $\ker(\phi)$ and 
 $\ker(\chi_{p})$, respectively) are linearly disjoint for ramification reasons.
 Hence $\pi(\cT')\supset (X^2\cdot \varLambda)^3$.

It follows that $\pi(\cT')$ and $\pi(\cT)$ are uniquely determined by their images in 
$\bar\Q_p[\epsilon]^3=(\varLambda/(X^2) )^3$, which we will now determine  using the tangent space computations from \S\ref{sec:tangent}. 

By  Chebotarev's density Theorem  $\cT'$ is generated over $\varLambda$ by the 
trace of $\rho_{\cT} = (\rho_{\cT_{\cusp}},\rho_{\eis}, \rho'_{\eis})$, hence the image of  $\pi(\cT')$ in 
$\bar\Q_p[\epsilon]^3$ is generated by $1$ and the image of the map $\xi :\G_\Q\to \bar\Q_p^3$ 
uniquely determined by 
\[
(\tr(\rho_{\cT,\epsilon})-\tr(\rho))=
(\det(\rho_{\cT,\epsilon})-\det(\rho))\cdot  \xi.
\]
Using  formulas \eqref{cusp-trace} and \eqref{system} for  $\rho_{\cT_{\cusp}}\mod{X^2}$, together with the formula \eqref{eq:eis-defn}  for $\rho_{\eis}$ yields:
 \begin{align}\label{eq:xi}
 \xi=\left( \frac{\cL(\phi^{-1})+\cL(\phi) \cdot \phi^{-1}}{\cL(\phi)+\cL(\phi^{-1})} , \mathbf{1}, \phi^{-1}     \right),
 \end{align}
which completes the proof of \eqref{eq:trace-image}. By \eqref{U_p equation1} the image of $U_p-1$ in $\bar\Q_p[\epsilon]^3$ 
belongs to $\bar\Q_p^\times(\epsilon,0,0)$, hence $\pi$ is an isomorphism. As  $\pi(X\cdot \gm_\cT)=(X^2\cdot \varLambda)^3$, one sees that  $X \cT$ is an ideal of $\cT'$, and the kernel of the natural surjective homomorphism $ (\cT'/X \cT)[U_p] \twoheadrightarrow   \cT/X \cT$ is generated by $(U_p-1)^2$.  
\end{proof}

\begin{rem}
 When $\phi$ is   quadratic,  the  relation in  \eqref{eq:trace-image} is given by the congruence 
  \[2a_\ell(\cF)\equiv a_\ell(\cE_{\mathbf{1}, \phi})+a_\ell(\cE_{\phi, \mathbf{1}})\pmod{X^2}\]  
  for all primes $\ell\nmid Np$. 
  This particular setting is a fruitful ground for arithmetic applications such as the proof by 
Bertolini,  Darmon and Venerucci  \cite{BeDaVe}  of a conjecture of Perrin-Riou relating the position of the  Kato class in the $\phi$-isotypic component of the Mordell-Weil group of an elliptic curve $E$ satisfying $L(E, \phi, 1)=0$   to a global point. 
Their work exploits the connection between that class and the generalised Kato class attached to $\phi$ seen as a genus character of its fixed  imaginary quadratic field, for which they prove  a formula mirroring  the above relation. 
\end{rem} 

\begin{rem}
 Using an algorithm based on methods of \cite{lauder}, 
A.~Lauder identified the linear relation \eqref{eq:xi} for  $\phi$ an odd sextic character of conductor $21$, and $p=13$ for which $\phi(13)=1$, which were used  in \cite[\S7.2]{DLR1} to provide a numerical evidence for  elliptic Stark points 
over cyclotomic fields.
\end{rem} 

 As a corollary of Theorem~\ref{prod-fibr} we investigate some ring theoretic properties of the
completed strict local ring $\cT$ of $\cC$  at $f$.  
 Examples of non-Gorenstein Hecke algebras abound in positive characteristic (see for example \cite{wang-wake}), but 
 seem to be  less common in characteristic zero.  

\begin{cor}\label{generalisedeigenspace}
\begin{enumerate}
\item The  ring $\cT$ is  Cohen--Macaulay, but  not Gorenstein. 
\item The local ring at $f$ of the fibre $\kappa^{-1}(\kappa(f))$ in $\cC$  has dimension  $3$ over $\bar\Q_p$.
\item The  ring $\cT'\simeq \varLambda[Y]/\left(Y(Y+\cL(\phi)X)(Y-\cL(\phi^{-1})X)\right)$ is  complete intersection.
\end{enumerate}
\end{cor}
\begin{proof}
(i)  The depth of $\cT$ cannot exceed its  Krull dimension which  is $1$. Since
$X=(X,X,X)\in \cT$ is a regular element,  the depth of $\cT$ is $1$ and, in particular the local ring $\cT$ is Cohen--Macaulay. 
By definition, $\cT$ is  Gorenstein if and only if its Artinian quotient $\cT/X\cT$ is  Gorenstein which is equivalent to its socle 
$\Hom_{\cT}(\bar\Q_p, \cT/X\cT)$ having dimension $1$ (see \cite[\S21.2, \S21.3]{eisenbud}). Since 
\[\cT\simeq \varLambda \times_{\bar\Q_p} \varLambda\times_{\bar\Q_p} \varLambda 
\simeq \bar{\Q}_p \lsem X_1,X_2,X_3 \rsem /(X_1X_2,X_1X_3,X_2X_3),\] 
where $X_1=(X,0,0)$, $X_2=(0,X,0)$ and  $X_3=(0,0,X)$, it follows that 
\[\cT/X\cT\simeq \bar{\Q}_p \lsem X_1,X_2 \rsem /(X_1^2,X_1X_2,X_2^2),\] 
hence its socle is $2$-dimensional. 

(ii) It follows from (i) that $\cT/\gm_{\varLambda}\cT=\cT/X\cT$ is a $\bar\Q_p$-vector space of dimension $3$.

(iii) By \eqref{eq:trace-image},    $\cT'$ is equidimensional of dimension $1$ and has $2$-dimensional  tangent space  generated by $X$ and $Y=(0,-\cL(\phi)X, \cL(\phi^{-1})X)$.  It follows that  $\cT' $ is a quotient of the 
 factorial ring  $\bar\Q_p\lsem X,Y \rsem$ by the element  $Y(Y+\cL(\phi)X)(Y-\cL(\phi^{-1})X)$, 
  in particular,   $\cT'$ is a  complete intersection. 
\end{proof}

\subsection{Duality for non-cuspidal Hida families}\label{sec:duality}
 In  Corollary~\ref{unique-cusp-family} we showed that $S^{\dag}_{\gm_f}=\varLambda\cdot \cF $. 
We will now exhibit  a basis of the free rank $3$ $\varLambda$-module $M^{\dag}_{\gm_{f}}$ (see Corollary~\ref{generalisedeigenspace}(ii)).

\begin{prop}\label{basishidafamilies} 
\begin{enumerate} 
\item The localisation of   $C_\varLambda$ at $\gm_f$  is free of rank $2$ over $\varLambda$.

\item The elements $E_{\mathbf{1},\phi}= (\cF-\cE_{\mathbf{1},\phi})/X$ and $E_{\phi,\mathbf{1}}=(\cF- \cE_{\phi,\mathbf{1}})/X$ are in $M^{\dag}_{\gm_{f}}$ and generate a complement of $S^{\dag}_{\gm_{f}}$, {\it i.e.} $M^{\dag}_{\gm_f}=\varLambda\cdot \cF\oplus\varLambda\cdot E_{\mathbf{1},\phi}\oplus\varLambda\cdot E_{\phi,\mathbf{1}}$. 
\end{enumerate}
\end{prop}

\begin{proof}
(i)  The freeness follows from  \eqref{residueniveau local} and the rank is $2$ as it is given  by the number of Eisenstein families containing  $f$. 

(ii)  Since all three families $\cF$, $\cE_{\mathbf{1},\phi }$ and $\cE_{\phi,\mathbf{1}}$ have coefficients in $\varLambda$ and specialise to $f$ in weight $\kappa(f)$, it follows  that both $E_{\mathbf{1},\phi}$ and $E_{\phi,\mathbf{1}}$ belong to $M^{\dag}_{\gm_{f}}$. By \eqref{residueniveau local} it suffices to show that $\res_\varLambda(E_{\mathbf{1},\phi})$ and
$\res_\varLambda(E_{\phi,\mathbf{1}})$ form a basis of  the localisation of   $C_\varLambda$ at $\gm_f$. 
This follows from Propositions~\ref{q-expa family}~and~\ref{ordL_p}, according to which  the constant term of $E_{\mathbf{1},\phi}$  at the cusp  $\infty$ belongs to $\varLambda^{\times}$ and vanishes at $0$, while $E_{\phi,\mathbf{1}}$ vanishes at  $\infty$, but belongs to  $\varLambda^{\times}$ at $0$.  \qedhere
 \end{proof}

 \begin{prop}\label{perfect-duality-totale} One has $\cT^{\full} \simeq \cT$  and the natural
 $\varLambda$-bilinear pairing:
 \[ \cT \times M^{\dag}_{\gm_f}\to \varLambda, \ \ (T, \cG)\mapsto a_1(T\cdot \cG)\]
is a  perfect duality. Its restriction to  $\cT^{\ord}_\rho \times (\varLambda\cdot \cF\oplus\varLambda\cdot E_{\mathbf{1},\phi})$ 
is a  perfect duality as well.
\end{prop}

\begin{proof} One has 
 $\cT^{\full} \simeq \cT$,   as $\cT^{\full}_{\cusp} \simeq \cT_{\cusp}$ 
 by Proposition~\ref{monodromy}, and for all $\ell\mid N$ one has $U_\ell(\cE_{\mathbf{1},\phi})=1$ and
 $U_\ell(\cE_{\phi,\mathbf{1}})=\chi_{p}(\Frob_\ell)\in 1+\gm_\varLambda$ (see \eqref{eq:univ-cyc-modX2}). 
 The non-degeneracy of the pairing follows from the $q$-expansion principle, as 
 $a_n(\cG)=a_1(T_n\cdot \cG)=0$ for all  $n\in\Z_{\geqslant 1}$.   The isomorphisms 
\[\cT \simeq \cT_{\cusp} \times_{\bar\Q_p}  \cT^{\eis}_\rho \times_{\bar\Q_p}  \cT^{\eis}_{\rho'} \simeq \varLambda \times_{\bar\Q_p} \varLambda\times_{\bar\Q_p} \varLambda,\]
established in Theorem~\ref{prod-fibr}, implies that the matrix of the pairing  in the bases 
$(\mathcal{F},E_{\mathbf{1},\phi},E_{\phi,\mathbf{1}})$ and $((1,1,1),(0,X,0),(0,0,X))$ is the identity,  hence the pairing  is a perfect duality.
\end{proof}

\begin{rem} As  $\cT^{\ord}_\rho$ is Gorenstein and finite flat over the regular ring $\varLambda$, it follows that the $\cT^{\ord}_\rho$-module $\Hom_{\varLambda}(\cT^{\ord}_\rho,\varLambda) \simeq M^{\dag}_{\gm_f,\rho}$ is  free of rank $1$. On the other hand, 
$\cT$ being  not Gorenstein, the $\cT$-module  $\Hom_{\varLambda}(\cT,\varLambda) \simeq M^{\dag}_{\gm_f}$ is not free of rank $1$. 
\end{rem}

\subsection{Non-classical overconvergent weight \texorpdfstring{$1$}{} modular forms}\label{sec:ocmf} 
Recall the universal cyclotomic character $\chi_{p}:\G_\Q\to \varLambda^\times$ from \S\ref{iwasawa-coh}, where $\varLambda=\Q_p \lsem X \rsem$ is  isomorphic to   the completed strict local ring  of the weight space $\cW$ at  $\kappa(f)$. 
 We will next compute infinitesimally the  $q$-expansion of the 
unique cuspidal  family  $\cF=\sum_{n\geqslant 1} {a_n(\cF)}q^n \in \varLambda \lsem q \rsem$
containing $f$ (see Corollary~\ref{unique-cusp-family}). 

\begin{prop}
\label{cuspidal-q-expansion} One has 
\begin{align}
\label{derivatives cusp2}
\left.\tfrac{d}{dX}\right|_{X=0}a_p(\cF) &= \frac{\cL(\phi)\cL(\phi^{-1}) }{(\cL(\phi)+\cL(\phi^{-1}))\log_p(1+p^\nu)}, \text{ and } 
\\
\label{derivatives cusp}
\left.\tfrac{d}{dX}\right|_{X=0}a_\ell(\cF)& =
\frac{(\phi(\ell)\cL(\phi^{-1})+\cL(\phi))\log_p(\ell)}{(\cL(\phi)+\cL(\phi^{-1}))\log_p(1+p^\nu)}, \text{ for every prime  } \ell \neq p. 
 \end{align}
\end{prop}

\begin{proof} Let  $\ell \nmid N$ be a prime. Reading the first component of \eqref{eq:xi} yields: 
\[\left.\tfrac{d}{dX}\right|_{X=0}a_\ell(\cF)=\left.\tfrac{d}{dX}\right|_{X=0}\tr(\rho_{\cT_{\cusp}})(\Frob_\ell)=
 \tfrac{\cL(\phi^{-1})+\cL(\phi) \phi^{-1}(\ell)}{\cL(\phi)+\cL(\phi^{-1})}\left.\tfrac{d}{dX}\right|_{X=0}\det(\rho_{\cT_{\cusp}})(\Frob_\ell).
 \]
As 
 $\det(\rho_{\cT_{\cusp}})=\phi\chi_{p}$ and $\left.\tfrac{d}{dX}\right|_{X=0} \chi_{p}(\Frob_\ell)= \frac{\log_p(\ell)}{\log_p(1+p^\nu)}$ by \eqref{derivative-chi-cyc},  we obtain  \eqref{derivatives cusp} for $\ell \nmid N$. 

In order to compute $\left.\tfrac{d}{dX}\right|_{X=0}a_p(\cF)=\left.\tfrac{d}{dX}\right|_{X=0}\chi_{\cT_{\cusp}}(\Frob_p)$ we go back to  the proof of Proposition~\ref{cuspidal tangent}.
Comparing  $\det (\rho_{\epsilon})=\phi(1+\epsilon (\lambda+\mu) \eta_\mathbf{1}) $ from  
\eqref{cusp-trace} with \eqref{derivative-chi-cyc} gives  $\lambda+\mu =\frac{-1}{\log_p(1+p^\nu)}$. 
Combining this with $ \mu \cL(\phi^{-1}) =\lambda \cL(\phi)$ from  \eqref{system} we obtain 
\begin{align}\label{eq:lambda-mu}
\lambda=-\frac{\cL(\phi^{-1})}{(\cL(\phi)+\cL(\phi^{-1}))\log_p(1+p^\nu)} \quad \text{ and } \quad
\mu=-\frac{\cL(\phi)}{(\cL(\phi)+\cL(\phi^{-1}))\log_p(1+p^\nu)}. 
\end{align}
 By \eqref{U_p equation1} one finds  $\left.\tfrac{d}{dX}\right|_{X=0}\chi_{\cT_{\cusp}}(\Frob_p)=-\mu \cL(\phi^{-1})= \tfrac{\cL(\phi)\cL(\phi^{-1})}{(\cL(\phi)+\cL(\phi^{-1}))\log_p(1+p^\nu)} $ as claimed.
 It remains to compute the derivative of $a_\ell(\cF)$ for primes $\ell | N$. From Proposition~\ref{monodromy}, $a_\ell(\cF)$ is given by the action of $\Frob_\ell$ on the $ (\rho_{\cT_{\cusp}})^{\I_{\Q_\ell}}$. By the proof of Proposition~\ref{cuspidal tangent} and \eqref{eq:lambda-mu} we get 
 \[
\left. \tfrac{d}{dX}\right|_{X=0}a_\ell(\cF)=\mu \eta_\mathbf{1}(\Frob_\ell)= \frac{ \cL(\phi) \log_p(\ell)}{(\cL(\phi)+\cL(\phi^{-1}))\log_p(1+p^\nu)},
 \]
 yielding \eqref{derivatives cusp} also for $\ell \mid N$, as in this case $\phi(\ell)=0$. 
 \end{proof}

Let $M^{\ord}_{\kappa(f)}$ be the space of ordinary overconvergent $p$-adic modular forms of weight $1$ and central character $\phi$. The eigenform $f$ corresponds to a maximal  ideal $\gm_f$ of the Hecke algebra   $\cT_{\kappa(f)}$ acting on $M^{\ord}_{\kappa(f)}$.    For  $i \geqslant 1$, let    $M^{\dag}[\gm_f^i]$ denote   the subspace of 
$M^{\ord}_{\kappa(f)}$  annihilated by $\gm_f^i$.

\begin{proof}[\bf Proof of Theorem~\ref{qexpansionoverc}]  
By definition the generalised eigenspace at $f$  is given by  the $\cT/\gm_\varLambda\cT$-module
$M^{\dag}_{\gm_f}/ \gm_\varLambda M^{\dag}_{\gm_f}$. 
By Theorem~\ref{prod-fibr}
 $\gm_f^2\cT \subset \gm_\varLambda\cT$, hence 
$M^{\dag} \lsem f \rsem=M^{\dag}[\gm_f^2]$. Note that by Proposition~\ref{perfect-duality-totale} one already knows that 
$\dim_{\bar\Q_p} M^{\dag}[\gm_f^2]= \dim_{\bar\Q_p} (\cT/\gm_f^2\cT)=3$.
By Proposition~\ref{basishidafamilies},  the specialisations of the families  $E_{\phi,\mathbf{1}}(X)$ and $E_{\mathbf{1},\phi}(X)$  at $X=0$  span a complement of  $S^{\dag}\lsem f \rsem$ in 
$M^{\dag}\lsem f \rsem$.   Consider 
 \begin{align}\label{defn-basis-ocmf}
\begin{split}
f^\dag_ {\mathbf{1},\phi}=&\tfrac{(\cL(\phi)+\cL(\phi^{-1}))\log_p(1+p^\nu)}{\cL(\phi)}  E_{\mathbf{1}, \phi}(0) 
  =\tfrac{(\cL(\phi)+\cL(\phi^{-1}))\log_p(1+p^\nu)}{\cL(\phi)}  \left.\tfrac{d}{dX}\right|_{X=0}\left(\cF-\cE_{\mathbf{1}, \phi}\right), \\
  f^\dag_{\phi, \mathbf{1}}=&\tfrac{(\cL(\phi)+\cL(\phi^{-1}))\log_p(1+p^\nu)}{\cL(\phi^{-1})}  E_{\phi, \mathbf{1}}(0) 
  =\tfrac{(\cL(\phi)+\cL(\phi^{-1}))\log_p(1+p^\nu)}{\cL(\phi^{-1})}  \left.\tfrac{d}{dX}\right|_{X=0}\left(\cF-\cE_{\phi,\mathbf{1}}\right).
\end{split}
\end{align}
 Since
  $a_p(\cE_{\mathbf{1},\phi})=a_p(\cE_{\phi,\mathbf{1}})=1$ 
  one finds that $ \left.\tfrac{d}{dX}\right|_{X=0}a_p(\cE_{\mathbf{1},\phi})=\left.\tfrac{d}{dX}\right|_{X=0}a_p(\cE_{\phi,\mathbf{1}})=0$.
  Further  
$a_\ell(\cE_{\mathbf{1},\phi})=1+\phi(\ell)\chi_{p}(\ell)$ and $a_\ell(\cE_{\phi,\mathbf{1}})=\phi(\ell)+\chi_{p}(\ell)$, together with  \eqref{derivative-chi-cyc} yields  that
\begin{align*}
\left.\tfrac{d}{dX}\right|_{X=0}a_\ell(\cE_{\mathbf{1},\phi})=\phi(\ell)\tfrac{\log_p(\ell)}{\log_p(1+p^\nu)}, \text{  and } \,\, 
 \left.\tfrac{d}{dX}\right|_{X=0}a_\ell(\cE_{\phi,\mathbf{1}})=\tfrac{\log_p(\ell)}{\log_p(1+p^\nu)}.
\end{align*}
Combining these formulas with  those given in Proposition~\ref{cuspidal-q-expansion}
 we obtain the desired formulas for the non-constant coefficients of $f^\dag_{\mathbf{1}, \phi}$ and $f^\dag_{ \phi, 1}$ defined in \eqref{defn-basis-ocmf}:  
\begin{align*}
a_p(f_{\mathbf{1}, \phi}^\dag)=\cL(\phi^{-1}), \, a_p(f_{\phi, \mathbf{1}}^\dag)=\cL(\phi) \text{ and }
a_\ell(f^\dag_{\mathbf{1},\phi})=(1-\phi(\ell))  \log_p(\ell)=-a_\ell(f^\dag_{\phi, \mathbf{1}}). 
\end{align*}

In order to compute  the remaining positive coefficients of  $f^\dag=f^\dag_{\mathbf{1}, \phi}$ or $f^\dag_{\phi,\mathbf{1}}$, we observe that since  
 $\cE_{\mathbf{1},\phi}$, $\cE_{\phi,\mathbf{1}}$ and  $\cF$  
are normalised eigenform for all Hecke operators $(T_n)_{n\geqslant 1}$ (see Proposition~\ref{monodromy} for $\cF$), the classical  relations between abstract Hecke operators imply:
\begin{align}
\label{coprime factors}
a_{mn}(f^\dag)&=a_m(f)a_n(f^\dag)+a_n(f)a_m(f^\dag),
 \text{ for   } (n,m)=1, \\
  \nonumber
 a_{\ell^r}(f^\dag)&= r a_\ell(f)^{r-1} a_{\ell}(f^\dag)=r a_{\ell}(f^\dag)  \text{ for  all primes } \ell \mid Np, r\geqslant 1, \text{ and }\\
 \nonumber a_{\ell^r}(f^\dag)&= a_\ell(f) a_{\ell^{r-1}}(f^\dag)+a_{\ell^{r-1}}(f)a_{\ell}(f^\dag) -
 \phi(\ell) a_{\ell^{r-2}}(f^\dag)  \text{ for   } \ell \nmid Np, r\geqslant 2.
\end{align}
 As $a_{\ell^r}(f)=\sum_{i=0}^r \phi(\ell)^i$ for all primes $\ell \neq p$, an induction on $r$ yields 
\[ \textstyle a_{\ell^r}(f^\dag)=a_{\ell}(f^\dag)\sum_{i=0}^{r}(i+1)(r-i)\phi(\ell)^i,  \text{ hence }\]
\[\textstyle a_{\ell^r}(f^\dag_{\mathbf{1},\phi})=\sum_{i=0}^{r}(r-2i)\phi(\ell^i)\log_p(\ell) =-a_{\ell^r}(f^\dag_{\phi, \mathbf{1}}), 
\text{ for  all primes } \ell \neq p.\]
Combining this with \eqref{coprime factors} yields the desired formulas for 
$a_n(f^\dag_{\mathbf{1}, \phi})$ and $a_n(f^\dag_{\phi, \mathbf{1}})$ for all $n\geqslant 1$. 

We have $a_0(f^\dag _{\phi, \mathbf{1}})=0$  because $a_0(\cE_{\phi, \mathbf{1}})$ is identically zero. 
To  compute  $a_0(f^\dag_{\mathbf{1}, \phi})$ we notice that by the above computations one has
$a_n\left(\frac{f^\dag_{\mathbf{1}, \phi}+f^\dag _{\phi, \mathbf{1}}}{\cL(\phi)+\cL(\phi^{-1})}\right)=a_n(E_1(\mathbf{1}, \phi)-f)$ for all $n\geqslant 1$, where $E_1(\mathbf{1}, \phi)$ is the classical Eisenstein series defined in \eqref{eq:def-eisenstein}. 
We then deduce from  Proposition~\ref{perfect-duality-totale} that 
$a_0\left(\frac{f^\dag_{\mathbf{1}, \phi}+f^\dag _{\phi, \mathbf{1}}}{\cL(\phi)+\cL(\phi^{-1})}\right)=a_0(E_1(\mathbf{1}, \phi)-f)$ as well, hence 
\begin{align} \label{first-formula}
a_0(f^\dag_{\mathbf{1}, \phi})=(\cL(\phi)+\cL(\phi^{-1}))\frac{L(\phi, 0)}{2}.      
\end{align} 
\end{proof}

\begin{cor}\label{Grossstarkcor}
$L'_p(\phi \omega_p, 0)=\cL(\phi)L(\phi, 0)$.
\end{cor}

\begin{proof} We will now compute $a_0(f^\dag_{\mathbf{1}, \phi})$ by a different method and obtain the 
formula for  the derivative of  the Kubota--Leopoldt $p$-adic $L$-function  (see \cite[(8)]{DDP}).  
By Proposition~\ref{q-expa family} one has $a_0(\cE_{\mathbf{1},\phi })=\tfrac{\zeta_\phi}{2}$. Since  by \eqref{Kubota--leopoldt} one has \[ \zeta_\phi((1+p^\nu)^{s-1}-1)=L_p(\phi \omega_p, 1-s),\]
taking derivatives  yields $\zeta_{\phi}'(0)=-\frac{L'_p(\phi \omega_p, 0) }{\log_p(1+p^\nu)}$. Using \eqref{defn-basis-ocmf} one finds: 
\begin{align*}
\tfrac{\cL(\phi)}{(\cL(\phi)+\cL(\phi^{-1}))}  a_0(f^\dag_{\mathbf{1}, \phi})=
-\log_p(1+p^\nu) \left.\tfrac{d}{dX}\right|_{X=0}a_0(\cE_{\mathbf{1},\phi })=
-\log_p(1+p^\nu) \tfrac{\zeta_{\phi}'(0)}{2}=\tfrac{L'_p(\phi \omega_p, 0)}{2}.
\end{align*}

From \eqref{first-formula} and  the latter we obtain 
$L'_p(\phi \omega_p, 0)=\cL(\phi)L(\phi, 0)$. 
\end{proof}

\begin{rem} \label{tame}
 In this closing remark we let  $p=5$,  $N=11$ and we denote by $\Delta$  the $5$-Sylow subgroup of  $(\Z/11\Z)^\times$.  
Mazur observed in \cite{mazur-eis} that the weight $2$ level $\Gamma_0(11)$ Eisenstein series is congruent modulo $5$ to 
the unique  weight $2$  cuspform $g$ of level  $\Gamma_0(11)$.   Merel \cite{merel} gave a numerical criterion for this  uniqueness  in terms of the non-vanishing of a tame derivative, in the sense of Mazur-Tate, of the zeta element  $\zeta_\Delta\in (\Z/5\Z)[\Delta]$  specialising  to $L(-1, \chi)$ for any  character  $\chi$ of $\Delta$ (considered as even Dirichlet character of conductor $11$). 
To draw a parallel with our current work, we consider the ordinary  $5$-stabilisation of $g$ which is congruent modulo  $5$ to each of the two 
$11$-stabilisations of the unique weight $2$ Eisenstein series of level $\Gamma_0(5)$. The Hecke $\Z_5$-algebra of level  $\Gamma_0(55)$ localised at the corresponding maximal ideal 
has structure analogous  to that of the  $\varLambda$-algebra $\cT$ described in Theorem~\ref{main-thm}(ii); in  
particular, it is not Gorenstein and   its cuspidal quotient is free of rank one over $\Z_5$. 
The analogy goes further, as the tame derivative   $\zeta'_\Delta=\sum_{\delta\in \Delta} a_\delta\log(\delta)[\delta]$  
evaluated at the $\chi=1$ does not vanish and its value differs from 
 $\zeta(-1)=\tfrac{1}{12}$ by a `tame $\cL$-invariant' given by $\log(4)$,  where 
$\log:\Delta\xrightarrow{\sim} \Z/5\Z$ is  a choice of  discrete  logarithm as in   \cite{lecouturier}. 
\end{rem}

\bibliographystyle{siam}

\end{document}